\documentclass[12pt]{amsart}
\usepackage{amssymb}
\usepackage{tikz}
\usetikzlibrary{arrows,matrix,positioning,automata,shapes,calc}
\usetikzlibrary{decorations.pathreplacing,calligraphy}

\usepackage{mathtools}
\usepackage{graphicx} 
\usepackage{xcolor}
\usepackage{tabstackengine}
\stackMath

\def\>{\relax\ifmmode\mskip.666667\thinmuskip\relax\else\kern.111111em\fi}
\def\:{\relax\ifmmode\mskip.333333\thinmuskip\relax\else\kern.0555556em\fi}
\def\<{\relax\ifmmode\mskip-.333333\thinmuskip\relax\else\kern-.0555556em\fi}
\def\?{\relax\ifmmode\mskip-.666667\thinmuskip\relax\else\kern-.111111em\fi}
\def\vsk#1>{\vskip#1\baselineskip}
\def\vv#1>{\vadjust{\vsk#1>}\ignorespaces}
\def\vvn#1>{\vadjust{\nobreak\vsk#1>\nobreak}\ignorespaces}
 \let\alb\allowbreak

\def\plait#1{\par\hangindent2\parindent\indent\kern\parindent
	\llap{#1\enspace}\ignorespaces}

\let\Smallskip\smallskip
\def\smallskip{\par\Smallskip}
\let\Medskip\medskip
\def\medskip{\par\Medskip}
\let\Bigskip\bigskip
\def\bigskip{\par\Bigskip}

\let\Maketitle\maketitle
\def\maketitle{\Maketitle\thispagestyle{empty}\let\maketitle\empty}

\newtheorem{thm}{Theorem}[section]
\newtheorem{cor}[thm]{Corollary}
\newtheorem{lem}[thm]{Lemma}
\newtheorem{prop}[thm]{Proposition}

\newtheorem{rem}[thm]{Remark}

\def\R{\mathbb R}
\def\C{\mathbb C}
\def\Z{\mathbb Z}

\def\gl{\mathfrak{gl}}

\def\g{\mathfrak{g}}

\let\on\operatorname

\def\diag{\on{diag}}
\def\End{\on{End}}

\def\Hom{\on{Hom}}
\def\id{\on{id}}

\def\Res{\on{Res}}

\def\gr{\on{gr}}

\def\qdet{\on{qdet}}

\let\alb\allowbreak
\def\lsym#1{#1\alb\dots\relax#1\alb} \def\lc{\lsym,}

\def\gln{\gl_n}
\def\Un{U(\gln)}
\def\Yn{Y(\gln)}

\def\glm{\gl_m}
\def\Um{U(\glm)}
\def\Ym{Y(\glm)}

\def\glnm{\gl_{n+m}}
\def\Unm{U(\glnm)}
\def\Ynm{Y(\glnm)}

\title{$(\gl_n,\gl_m)$-duality and Olshanski homomorphism}
\author[B\<.\,Feigin]{B\<.\,Feigin$\:^\diamond$}
\thanks{$\kern-\parindent^\diamond$ E\:-mail: borfeigin@gmail.com}

\author[L\<.\,Rybnikov]{L\<.\,Rybnikov$\:^\circ$}
\thanks{$\kern-\parindent^\circ$ E\:-mail: leonid.rybnikov@umontreal.ca}

\author[F\<.\,Uvarov]{F\<.\,Uvarov$\:^\star$}
\thanks{\noindent$^\star$ Corresponding author, E\:-mail: fuvarov@hse.ru}

\begin{document}

\maketitle

\begin{center}
\vsk-.2>
{\it $^{\diamond\:\star}\?$  National Research University Higher School of Economics, Faculty of Mathematics, \\ 
6 Usacheva str., Moscow, 119048, Russia\/}

\medskip
{\it $^\circ\?$ Department of Mathematics and Statistics,
University of Montreal,\\ Montreal QC, Canada \/}

\medskip
{\it $^\diamond\?$ Hebrew University of Jerusalem, Einstein Institute of Mathematics, \\ Givat Ram. Jerusalem, 9190401, Israel  \/}
\end{center}

\begin{abstract}
    We show that the images of the Bethe subalgebras of the Yangians $\Yn$ and $\Ym$ under the homomorphisms to $U(\gl_{n+m})$ given by the Olshanski centralizer construction coincide. We use this result to obtain the $(\gl_n,\gl_m)$-duality of the trigonometric Gaudin model and the XXX-spin chain. The duality is obtained in an explicit way relating the generating differential operator on one side and the generating difference operator on the other, thus agreeing with the result of Mukhin, Tarasov and Varchenko \cite{MTV4}.
\end{abstract}

\section{Introduction}
\subsection{}\label{1.1} Two quantum integrable systems, one associated to the Lie algebra $\gln$, the other to the Lie algebra $\glm$ are called $(\gln, \glm)$-dual in the following situation. They have the same space of states carrying both $\gln$- and $\glm$-symmetry, and the algebras of their quantum integrals of motion (called \emph{Bethe algebras}) in this representation coincide. 

The systems that are involved (or expected to be involved) in the $(\gln,\glm)$-duality can be included in the following diagram.
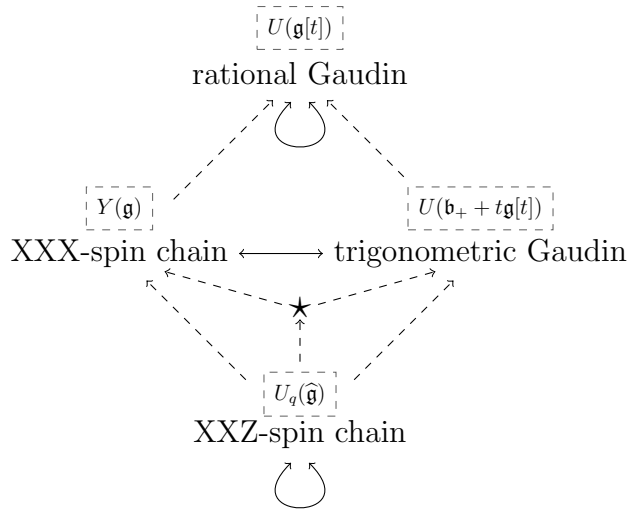
\vspace{1cm}

\begin{center}
\begin{tikzpicture}[scale=1.2, 
roundnode/.style={circle, draw=green!60, fill=green!5, very thick, minimum size=7mm},
squarednode/.style={rectangle, draw=black!60, dashed, minimum size=5mm},
]
    \node(1) at (0,0) {rational Gaudin};
    \node(1a) at (0.1,-0.2) {};
    \node(1b) at (-0.1,-0.2) {};
    \draw [<->] (1b.south) to[out=-135, in=-45,looseness=12] (1a.south); 
    \node(2) at (-2,-2) {XXX-spin chain};
    \node(3) at (2,-2) {trigonometric Gaudin};
    \draw [<->] (2) to (3); 
    \node(4) at (0,-4) {XXZ-spin chain};
    \node(4a) at (0.1,-4.2) {};
    \node(4b) at (-0.1,-4.2) {};
    \draw [<->] (4b.south) to[out=-135, in=-45,looseness=12] (4a.south);
    \node(4c) at (0.5,-3.8) {};
    \node(4d) at (-0.5,-3.8) {};
    \node(1c) at (0.5,-0.2) {};
    \node(1d) at (-0.5,-0.2) {};
    \node[squarednode] at (0,0.5) {\scalebox{0.7}{$U(\g [t])$}};
    \node[squarednode] at (-2,-1.5) {\scalebox{0.7}{$Y(\g)$}};
    \node[squarednode] at (2,-1.5) {\scalebox{0.7}{$U(\mathfrak{b}_{+}+t\g [t])$}};
    \node[squarednode] at (0,-3.56) {\scalebox{0.7}{$U_{q}(\widehat{\g})$}};
    \node(10) at (0,-2.6) {\scalebox{1.4}{$\star$}};
    \draw [dashed, ->] (0,-3.2) to (0,-2.73);
    \node(2a) at (-1.7,-2) {};
    \node(3a) at (1.7,-2) {};
    \draw [dashed, ->] (10.north)+(-0.2,-0.2) to (-1.5,-2.2);
    \draw [dashed, ->] (10.north)+(0.2,-0.2) to (1.5,-2.2);
    \draw [dashed, ->] (-1.4,-1.4) to (-0.3,-0.3);
    \draw [dashed, ->] (1.15,-1.15) to (0.3,-0.3);
    \draw [dashed, ->] (-0.6,-3.4) to (-1.7,-2.3);
    \draw [dashed, ->] (0.6,-3.4) to (1.7,-2.3);
\end{tikzpicture}
\end{center}
In the dashed boxes, we indicate the algebras, which are crucial for the construction of the corresponding systems, that is, the space of states is naturally a module over this algebra and the quantum integrals of motion come from some universal commutative (Bethe) subalgebra in it. The dashed arrows on the diagram indicate that one integrable system can be degenerated to another, and the solid double arrows connect the dual systems if for one system in the dual pair, we take $\g=\gln$, and for another, $\g=\glm$. Therefore, the duality corresponds to the reflection of the diagram with respect to the vertical line passing through the middle. In this paper, we introduce a system that is denoted by the symbol {\scalebox{1.4}{$\star$}} on the diagram and show that it is self-dual. Then we consider a degeneration of our result and establish the duality between the XXX-spin chain and the trigonometric Gaudin model.

\subsection{} First observations concerning the $(\gln, \glm)$-duality in quantum integrable models were made by V. Tarasov and A. Varchenko in their work \cite{TV4}, where it was shown that if we consider the $\gln\oplus\glm$-module $S(\C^{n}\otimes \C^{m})$, then the images of the rational Gaudin Hamiltonians associated with $\gln$ coincide with the images of the Dynamical Hamiltonians associated with $\glm$. Since the Dynamical Hamiltonians are elements in the Bethe algebra of the Gaudin model, this was the first indication of the self-duality of the rational Gaudin model that was established later in \cite{MTV1}. In the work \cite{TV4}, the authors also observed a similar relationship between the trigonometric Gaudin Hamiltonians and the XXX-Dynamical Hamiltonians, which suggested the duality between the trigonometric Gaudin model and the XXX-spin chain. The proof of the latter duality is one of the new results that we present in this paper. 

\subsection{} The Hamiltonians of all three systems mentioned in the previous paragraph can be diagonalized using the Bethe ansatz method. In the framework of this method, one starts with a vector space of functions in a variable $u$ with prescribed singularities and behavior around these singularities and uses a certain procedure to construct an eigenvector, see \cite{MV}. It was discovered in \cite{MTV2} that a version of the Fourier transform provides a bijection between spaces describing eigenvectors in $S(\C^{n}\otimes \C^{m})$ of rational $\gln$-Gaudin Hamiltonians and dual to them $\glm$-Dynamical Hamiltonians.  This observation motivated the main result of the paper \cite{MTV1}, which states that the generating differential operator of the image of the Bethe algebra of the rational $\gl_{n}$-Gaudin model in $\End(S(\C^{n}\otimes \C^{m}))$  can be obtained from the similar differential operator for the rational $\gl_{m}$-Gaudin model by applying the following antiautomorphism of the algebra of linear differential operators in $u$ with polynomial coefficients:
\begin{equation*}
    \partial_{u}\mapsto u,\quad u\mapsto\partial_{u},
\end{equation*}
where $\partial_{u}$ is the derivative with respect to $u$. 

One of the main results of the current paper, Theorem \ref{main3}, establishes a similar relationship between the generating differential operator for the Bethe algebra of the trigonometric $\gln$-Gaudin model and the generating difference operator for the Bethe algebra of the XXX-spin chain associated to $\glm$. The corresponding antihomomorphism is given by
\begin{equation*}
    u\partial_{u}\mapsto u,\quad u\mapsto\tau,
\end{equation*}
where $\tau$ is the shift operator: $(\tau f)(u)=f(u+1)$. This result is motivated by the work \cite{MTV4}, where it was shown that a version of the Mellin transform provides a bijection between spaces of functions describing eigenvectors in $S(\C^{n}\otimes \C^{m})$ of trigonometric $\gln$-Gaudin Hamiltonians and dual to them XXX-Dynamical Hamiltonians associated to $\glm$. 

We deduce Theorem \ref{main3} as a degeneration of a $(\gln,\glm)$-duality statement for a new quantum integrable system coming from the Olshanski centralizer construction. 

\subsection{} The integrals of our new integrable system come from the commutative subalgebra $\mathcal{A}(\chi)\subset \Unm$ which lies in invariants of the adjoint action of $GL_n\times GL_m$. Namely, this is the quantum \emph{shift of argument subalgebra} of \cite{Ryb2} with the Cartan parameter $\chi = \diag(1,\ldots,1,0,\ldots,0)$. By the general theorem of Olshanski \cite{Olsh,Mol1}, the subalgebra of adjoint $GL_n\times GL_m$-invariants in $\Unm$ is the image of $GL_n$-invariants in the Yangian $\Yn$, and, at the same time, the image of $GL_m$-invariants $\Ym$ under some explicit homomorphisms. Moreover, by \cite{IR}, the images of \emph{Bethe subalgebras}, $\mathcal{B}_{n}$ in $\Yn$ and $\mathcal{B}_{m}$ in $\Ym$, lie in $\mathcal{A}(\chi)$. The generators of the Bethe subalgebra in the Yangian $\Yn$ are coefficients of the difference operator $D_n$ on the affine line. We consider the difference operators $\overline{D}_{n}$ and $\overline{D}_{m}$ with the coefficients in $\mathcal{A}(\chi)\subset\Unm$ being the (appropriately normalized) images of $D_n$ and $D_m$, respectively, under the Olshanski homomorphisms to $\Unm$. 

Our main result is Theorem \ref{main1}, which identifies two difference operators $\overline{D}_{n}$ and $\overline{D}_{m}$ via an antihomomorphism given by 
 \begin{equation}\label{difference transform}
 \begin{split}
      & (u-n+1)d  \mapsto (u+n), \\
      & u+m \mapsto (u-m+1)d,
 \end{split}
 \end{equation}
 where $d=\tau-1$.

 \subsection{} After establishing the duality on the level of $\Unm$ in Theorem \ref{main1}, we consider the action of $\Unm$ on the Verma module $M_{\lambda}$ of the generic highest weight $\lambda$. The upper-left and lower-right block Lie subalgebras $\gln$ and $\glm$ act on $M_{\lambda}$. The coefficients of the difference operators $\overline{D}_{n}$ and $\overline{D}_{m}$ act on the subspace $M^{sing}_{\lambda}[\nu,\mu]$ of vectors of $\gln$-weight $\nu$ and $\glm$-weight $\mu$, which are singular with respect to both $\gln$ and $\glm$. Using the representation theory of Yangians, we construct an isomorphism of vector spaces $M^{sing}[\nu,\mu]\cong S^{(n)}(b)[a]$, where $S^{(n)}(b)[a]$ is a $\gln$-weight subspace of weight $a=(a_{1}\lc a_{n})$ of the $\gln$-module
 \[S^{(n)}(b)=S^{b_{1}}\C^{n}\otimes S^{b_{2}}\C^{n}\otimes\dots\otimes S^{b_{m}}\C^{n}\]
 for some $a$ and $b=(b_{1}\lc b_{m})$ depending on $\lambda$, $\mu$, and $\nu$. We also construct an isomorphism $M^{sing}[\nu,\mu]\cong S^{(m)}(a)[b]$, where $S^{(m)}(a)[b]$ is defined similarly to $S^{(n)}(b)[a]$. Combining the two isomorphisms together, we get an isomorphism $\Psi_{\lambda}: S^{(m)}(a)[b]\rightarrow S^{(n)}(b)[a]$. 

 The map $\Psi_{\lambda}$ should be considered as a deformation of the natural isomorphism $\Psi: S^{(m)}(a)[b]\rightarrow S^{(n)}(b)[a]$ given via the identification of the $\glm$-submodule $S^{(m)}(a)[b]$ and $\gln$-submodule $S^{(n)}(b)[a]$ inside $S(C^{n}\otimes\C^{m})$. It would be interesting to relate $\Psi_{\lambda}$ to the construction in the work \cite{Zhang} on $(U_{q}(\gln), U_{q}(\glm))$-duality. 

 The duality on the level of $\Unm$ implies that the map $\Psi_{\lambda}$ intertwines the action of a commutative subalgebra of $\End (S^{(m)}(a)[b])$ and the action of a similar commutative subalgebra in $\End (S^{(n)}(b)[a])$, see Theorem \ref{duality rep} in the text. The elements of these commutative subalgebras are quantum integrals of the self-dual quantum integrable system we denoted by {\scalebox{1.4}{$\star$}} in the above diagram. It is obtained as a deformation of the XXX-model, similar to the deformation, which one uses to obtain trigonometric Gaudin model from the rational Gaudin model. Namely, we can consider the universal Gaudin Bethe algebra $\mathcal{B}^{G}_{n}\subset U(\gln[t])$, then the Hamiltonians of the rational Gaudin model are given as elements of the projection of $\mathcal{B}_{n}^{G}$ under its action on certain finite-dimensional $U(\gln [t])$-module $V$ (which is usually a tensor product of evaluation modules), see \cite{Tal}, \cite{MTV6}. On the other hand, we can consider the action of $\mathcal{B}_{n}^{G}$ on $V\otimes M^{(n)}_{\lambda^{(n)}}$, where $M^{(n)}_{\lambda^{(n)}}$ is the $\gln$-Verma module of generic highest weight $\lambda^{(n)}$, and restrict it to the subspace $(V\otimes M^{(n)}_{\lambda^{(n)}})^{sing}[\nu]$ of singular vectors of weight $\nu$. We have an isomorphism $(V\otimes M^{(n)}_{\lambda^{(n)}})^{sing}[\nu]\cong V[\nu-\lambda^{(n)}]$, and it is known that the trigonometric Gaudin Hamiltonians in $\End (V[\nu-\lambda^{(n)}])$  are given as elements of the projection of $\mathcal{B}_{n}^{G}$ to $\End ((V\otimes M^{(n)}_{\lambda^{(n)}})^{sing}[\nu])$, see \cite{IKR}, and Section \ref{trigonometric Gaudin} below. If we consider the similar construction using the universal XXX-Bethe algebra instead of $\mathcal{B}_{n}^{G}$, which is the Yangian Bethe algebra $\mathcal{B}_{n}$ mentioned above, we get the system {\scalebox{1.4}{$\star$}}, whose self-duality is established in Theorem \ref{duality rep} using the map $\Psi_{\lambda}$. The relation of this system to the XXZ-spin chain model is not clear for us yet, but we expect that the former is a degeneration of the latter.

 \subsection{}
 After establishing the self-duality on the level of representations in Theorem \ref{duality rep}, we consider the following limit. Let $\lambda = (\lambda_{1}\lc \lambda_{n+m})$ be the highest weight of the $\glnm$-Verma module $M_{\lambda}$ mentioned above. Then for each $i=n+1\lc n+m$, we replace $\lambda_{i}$ with $r\xi_{i}$ for some new parameters $r$ and $\xi_{i}$ and take the limit $r \rightarrow \infty$. We show that in this limit, the projection of the difference operator $\overline{D}_{m}$ defines the generating difference operator of the XXX-Bethe algebra. We also show that in this limit, the projection of the difference operator $\overline{D}_{n}$ becomes the generating \emph{differential} operator of the trigonometric Gaudin model if prior to taking the limit, we also replace $u$ with $ru$ and $\tau$ with $\exp (r^{-1} \partial_{u})$.

 The most technically involved part of Section \ref{degeneration} deals with the limit of the map $\Psi_{\lambda}$. We first have to modify this map to make the corresponding limit invertible. Then we use the fact that for generic values of the parameters, the trigonometric Gaudin and XXX dynamical Hamiltonians are diagonalizable with one-dimensional common eigenspaces, and compare the limit of  the modified $\Psi_{\lambda}$ with the isomorphism $\Psi$, which we need for the duality of trigonometric Gaudin and XXX models, formulated in Theorem \ref{main3}.

 \subsection{} We conclude the introduction with a brief review of some other known results related to the the $(\gln, \glm )$-duality in quantum integrable systems. Such a duality was studied in the setting of quantum toroidal algebras, see \cite{FJM}. The self-duality of the rational Gaudin model was generalized to the case, when $\glm$ is replaced with the super Lie algebra $\gl_{m|k}$, see \cite{HM1}. We expect that our result can be also generalized to the case of super Lie algebras. The $(\gln, \glm)$-duality of quantum models via the quantum-classical correspondence \cite{MTV9}, \cite{GZZ}, \cite{KorZei} is related to the so-called Ruijsenaars duality of classical integrable models, see \cite{PZ}. Under the quantum-classical correspondence, the Gaudin model is interchanged with the Calogero model, and spin chains are interchanged with Ruijsenaars models. Since the Ruijsenaars duality is expected to involve elliptic models as well, this suggests that the $2\times 2$ square diagram presented in Subsection \ref{1.1} can be extended to a $3\times 3$ square diagram.

As indicated in the work \cite{RSVZ}, the correspondence of the Bethe vectors in some special cases may have a geometric interpretation via the 3d-mirror symmetry. It would be interesting to see how the equalities of the images of the Bethe algebras can be interpreted in these cases. 

\subsection{} The paper is organized as follows. In Section \ref{Sec2}, we recall some facts about Yangians and Bethe algebras. In Section \ref{Sec3} we study the images of the Yangian Bethe algebras in $\Unm$ under the Olshanski homomorphism and prove that the generating difference operators $\overline{D}_{n}$ and $\overline{D}_{m}$ are related via the antihomomorphism \eqref{difference transform}. In Section \ref{Sec4}, we consider the action of the Lie algebra $\glnm$ on the Verma module $M_{\lambda}$ and use the representation theory of Yangians to build the isomorphism $\Psi_{\lambda}: S^{(m)}(a)[b]\rightarrow S^{(n)}(b)[a]$. We then show that from the construction of $\Psi_{\lambda}$ and the results of Section \ref{Sec3}, it follows that $\Psi_{\lambda}$ intertwines the actions of the Bethe algebras of new integrable systems. In Section \ref{degeneration}, we consider the limit of the construction from Section \ref{Sec4}, and establish the duality between the trigonometric Gaudin and XXX-spin chain.

\subsection{Acknowledgements} We thank Evgeny Mukhin and Vitaly Tarasov for discussions and comments. B.F. is partially supported by ISF 3037 2025 1 1. L.R. is supported by the Courtois Foundation. The work of F. U. is an output of a research project implemented as part of the Basic Research Program at the National Research University Higher School of Economics (HSE University).

\section{Preliminaries}\label{Sec2}
\subsection{Basics on Yangians and their representations}
For any associative algebra $\mathcal{A}$, let us denote the algebra of Laurent series in $u$ with coefficients in $\mathcal{A}$ as $\mathcal{A}((u))$. We will identify any rational function $f\in\C (u)$ and its Laurent series at infinity, that is, we will think of $\mathcal{A}(u)=\mathcal{A}\otimes\C (u)$ as a subalgebra of $\mathcal{A}((u^{-1}))$. Also, let us identify any homomorphism of associative algebras $h:\mathcal{A}\rightarrow \mathcal{B}$ and its natural extensions to $\mathcal{A}((u))$ and $\mathcal{A}\otimes \End(\C^{n})$.

The Yangian $\Yn$ is the associative unital algebra with generators $t_{ij}^{(r)}$, $i,j=1\lc n$, $r\in\Z_{>0}$, and relations defined as follows. Introduce the generating series $t_{ij}(u)\in \Yn ((u^{-1}))$:
\[t_{ij}(u)=\delta_{ij}+\sum_{r=1}^{\infty}t_{ij}^{(r)}u^{-r}.\]
Let $\{e_{i},\, i=1\lc n\}$ be the standard basis of $\C^{n}$, that is, $e_{i}=(0\lc 0,1,0\lc 0)$, where $1$ is on the $i$-th place. For each $i,j=1\lc n$, define $E_{ij}\in\End (\C^{n})$ by $E_{ij}e_{j} = e_{i}$. Introduce the elements $T_{1}(u)$, $T_{2}(v)$ and $R_{12}(u,v)$ of the tensor product $\End (\C^{n})\otimes\End(\C^{n})\otimes \Yn((u^{-1}))((v^{-1}))$:
\[T_{1}(u)=\sum_{i=1}^{n}\sum_{j=1}^{n}E_{ij}\otimes\id\otimes t_{ij}(u),\quad T_{2}(v)=\sum_{i=1}^{n}\sum_{j=1}^{n}\id\otimes E_{ij}\otimes t_{ij}(v),\]
\[R_{12}(u,v)=u-v+\sum_{i=1}^{n}\sum_{j=1}^{n}E_{ij}\otimes E_{ji}\otimes 1.\]
Then the relations in $\Yn$ are given by the following formula:
\begin{equation}\label{RTT}
R_{12}(u-v)T_{1}(u)T_{2}(v)=T_{2}(v)T_{1}(u)R_{12}(u-v).
\end{equation}

Let $\{e_{ij}\,\vert\,i,j=1\lc n\}$ be the standard basis of the Lie algebra $\gln$, that is, the image of $e_{ij}$ under the action of $\gln$ on $\C^{n}$ is the operator $E_{ij}$. Then we have the evaluation homomorphism $ev_{n}:\Yn\rightarrow \Un$ given by  
\[ev_{n}: t_{ij}(u)\mapsto \delta_{ij}+\frac{e_{ij}}{u}.\]
Also, for any $c\in\C$, we have the shift homomorphism $sh_{c}:\Yn\rightarrow\Yn$, $sh_{c}(t_{ij}(u))=t_{ij}(u-c)$.
If $V$ is a $\Un$-module, let $V(c)$ denote the $\Yn$-module obtained by pulling back $V$ via $sh_{c}\circ ev_{n}$. The module $V(c)$ is called an evaluation module with evaluation parameter $c$.

Let $W$ be a $\Yn$-module. We say that a vector $w\in W$ has a $\Yn$-weight $(\lambda_{1}(u),\lambda_{2}(u)\lc \lambda_{n}(u))\in\bigl(\C((u^{-1}))\bigr)^{n}$ if for any $i=1\lc n$, we have $t_{ii}(u)w=\lambda_{i}(u) w$. Then $W$ is called a $\Yn$-module of highest $\Yn$-weight $\overline{\lambda}(u)\in\bigl(\C((u^{-1}))\bigr)^{n}$ if $W=\Yn\, v$ for some $v$ that has weight $\overline{\lambda}(u)$ and such that $t_{ij}(u)v=0$ for any $i<j$. For any $\overline{\lambda}(u)\in\bigl(\C((u^{-1}))\bigr)^{n}$, there exists unique up to isomorphism irreducible $\Yn$-module of highest $\Yn$-weight $\overline{\lambda}(u)$.  

\subsection{Bethe subalgebras of Yangians}\label{Bethe subalgebras in Yangian} 


For any associative unital algebra $\mathcal{A}$, let $D(\mathcal{A})$ denote the algebra of difference operators in the variable $u$ with coefficients in $\mathcal{A}((u^{-1}))$, that is, $D(\mathcal{A})$ is the unital subalgebra of $\End\bigl(\mathcal{A}((u^{-1}))\bigr)$ generated by $\mathcal{A}((u^{-1}))$ and the shift operator $\tau$ given by 
\[(\tau\cdot f) (u) = f(u+1),\quad f(u)\in \mathcal{A}((u^{-1})).\] 
Here and throughout the paper, we identify $\mathcal{A}((u^{-1}))$ with its image in $\End\bigl(\mathcal{A}((u^{-1}))\bigr)$ under the regular representation and use the following notation: for any $D\in D(\mathcal{A})$ and $f\in\mathcal{A}((u^{-1}))$, $D\cdot f\in\mathcal{A}((u^{-1}))$ denotes the result of the action of $D$ on $f$, while $Df\in D(\mathcal{A})$ is just a product of $D$ and $f$ in $D(\mathcal{A})$. We will also identify any homomorphism of associative algebras $h:\mathcal{A}\rightarrow \mathcal{B}$ and its natural extension to $D(\mathcal{A})$.

Fix a sequence of complex numbers $C=(c_{1}\lc c_{n})$, and define
$D_{n}(C)\in D(\Yn)$ as follows:

\begin{equation}\label{cdet} 
D_{n}(C) = \sum_{\sigma\in S_{n}}(-1)^{\sigma}D_{\sigma (1),1}D_{\sigma (2), 2}\dots D_{\sigma (n), n},
\end{equation}
where
\begin{equation}\label{mDC}
D_{ij} = c_{i}\delta_{ij}\tau-t_{ij}(u-j+1).
\end{equation}
Let $B_{ij}(C)\in\Yn$, $i=0\lc n$, $j\in\Z_{\geq 0}$ be the coefficients of $D_{n}(C)$:
\begin{equation}\label{DC}
D_{n}(C)=\sum_{i=0}^{n}(-1)^{i}\left(\sum_{j=0}^{\infty}B_{ij}(C)u^{-j}\right)\tau^{n-i}.
\end{equation}

Define the Bethe algebra $\mathcal{B}_{n}(C)$ to be the unital subalgebra of $\Yn$ generated by $B_{ij}(C)\in\Yn$, $i=0\lc n$, $j\in\Z_{\geq 0}$. It is known, see \cite{KS}, \cite{MTV6}, that $\mathcal{B}_{n}(C)$ is commutative. If $C=(1,1,\dots ,1)$, we will write $\mathcal{B}_{n}$, $B_{ij}$, and $D_{n}$ for $\mathcal{B}_{n}(C)$, $B_{ij}(C)$ and $D_{n}(C)$, respectively. 

Denote $B_{i}(u, C)=\sum_{j=0}^{\infty}B_{ij}(C)u^{-j}$.
For each $J=\{j_{1}<j_{2}<\dots <j_{k}\}\subset\{1\lc n\}$, introduce the quantum minor $t_{J}(u)$ as follows:
\begin{equation}\label{qminor}
t_{J}(u)=\sum_{\sigma\in S_{k}}(-1)^{\sigma}t_{j_{\sigma(1)},j_{1}}(u)t_{j_{\sigma(2)},j_{2}}(u-1)\dots t_{j_{\sigma(k)},j_{k}}(u-k+1).
\end{equation}
It is easy to see that
\begin{equation}\label{B=t}
       B_{k}(u, C)=\sum_{J,\,|J|=k}\left(\prod_{j\notin J}c_{j}\right) t_{J}(u),
\end{equation}
In particular, $B_{n}(u, C)$ does not depend on $C$. It is known, see \cite{KS} \cite{Mol2}, that $B_{n}(u, C)$ is a generating series for the center of $\Yn$. It is called the quantum determinant, and we will denote it as $\qdet_{n}(u)$.

\section{The duality on the level of the universal enveloping algebra $\Unm$}\label{Sec3} 
\subsection{Olshanski homomorphism}\label{Oh}
The element $T(u)\in\End(\C^{n})\otimes\Yn((u^{-1}))$ defined by
\[T(u)=\sum_{i=1}^{n}\sum_{j=1}^{n}E_{ij}\otimes t_{ij}(u)\]
is invertible. For each $i.j=1\lc n$, define $t'_{ij}(u)$ by $T^{-1}(u)=\sum_{i,j=1}^{n}E_{ij}\otimes t'_{ij}(u)$. Write $t'_{ij}(u)=\sum_{k=0}^{\infty}(t')^{(k)}_{ij}u^{-k}$. It is evident from relations \eqref{RTT} that $\omega_{n}: t^{(k)}_{ij}\mapsto (-1)^{k}(t')^{(k)}_{ij}$ defines an automorphism of $\Yn$.

Introduce the following inclusions:
\begin{align*}
    & i_{n}: \Yn\hookrightarrow \Ynm,\quad  i_{n}(t_{ij}(u))=t_{ij}(u), \\
    & j_{m}: \Ym\hookrightarrow \Ynm, \quad  j_{m}(t_{ij}(u))=t_{n+i,n+j}(u).
\end{align*}

We will be interested in the following compositions (which are also inclusions):
\[\phi_{n}=\omega_{n+m}\circ i_{n}\circ\omega_{n}:\Yn\hookrightarrow \Ynm,\quad \psi_{m}=\omega_{n+m}\circ j_{m}\circ\omega_{m}:\Ym\hookrightarrow \Ynm.\]

The map $\psi_{m}$ is closely related to the centralizer construction of $\Yn$ presented by G. Olshanski in the work \cite{Olsh}, and we call it the Olshanski homomorphism, and $\phi_{n}$ is a similar map, which can be obtained from $\psi_{n}$ using the automorphism $\delta_{n}$: $\Yn\rightarrow\Yn$, $t_{ij}(u)\mapsto t_{n-j+1,n-i+1}(-u)$ as follows: 
\begin{equation}\label{phi delta}
    \phi_{n}\delta_{n}=\delta_{n+m}\psi_{n}.
\end{equation}
The last formula is an immediate consequence of 
\begin{equation}\label{i delta}
    i_{n}\delta_{n}=\delta_{n+m}j_{n}
\end{equation}
 and  $\omega_{n}\delta_{n}=\delta_{n}\omega_{n}$, which can be checked directly.

For any associative algebra $\mathcal{A}$ and $l\in\Z$, introduce a linear map $\widehat{\delta}_{l}: D(\mathcal{A})\rightarrow D(\mathcal{A})$ as follows:
\[\widehat{\delta}_{l}: \sum_{i}b_{i}(u)\tau^{i} \mapsto \sum_{i} \tau^{i} b_{i}(-u+l-1). \]
The quantum minors introduced in formula \eqref{qminor} can be written in the following form:
\[t_{J}(u)=\sum_{\sigma\in S_{k}}(-1)^{\sigma}t_{j_{1},j_{\sigma(1)}}(u-k+1)t_{j_{2},j_{\sigma(2)}}(u-k+2)\dots t_{j_{k},j_{\sigma(k)}}(u),\]
see \cite[formula (1.55)]{Mol2}.
Using that, one can check that for any $J=\{j_{1}\lc j_{k}\}\subset\{1\lc n\}$, we have 
\begin{equation}\label{delta t}
\delta_{n}(t_{J}(u))=t_{J'}(-u+k-1),
\end{equation}
where $J'=\{n-j_{1},n-j_{2}\lc n-j_{k}\}$. Together with formula  \eqref{B=t}, this gives 
\begin{equation}\label{delta D}
    \delta_{n}(D_{n})=\widehat{\delta}_{n}(D_{n}).
\end{equation}

\begin{lem}\label{lem1}
    Denote $C_{1,0}=(1,1\lc 1,0,0\lc 0)$ with $n$ ones and  $C_{0,1}=(0,0\lc 0,1,1\lc 1)$ with $m$ ones. Then 
    \begin{equation}\label{psi D}
        i_{n}(\qdet_{n}(u+n))\,\psi_{m}(D_{m})=(-1)^{n}\tau^{n}D_{n+m}(C_{0,1})\tau^{-n}.
    \end{equation}
    \begin{equation}\label{phi D}
        j_{m}(\qdet_{m}(-u-1))\,\phi_{n}(\widehat{\delta}_{n}D_{n})=(-1)^{m}\tau^{m}\widehat{\delta}_{n+m}(D_{n+m}(C_{1,0}))\tau^{-m},
    \end{equation}    
\end{lem}
\begin{proof}
 The proof of \eqref{psi D} is a straightforward check using \eqref{B=t} and a well-known (see \cite[Theorem 1.10.7]{Mol2}) relation
 \begin{equation}\label{sylvester}
 \qdet_{n}(u)\cdot\omega_{n}(t_{J}(-u+n-1))=t_{J^{c}}(u),
 \end{equation}
 where $J^{c}=\{1\lc n\}\setminus J$.

 Relation \eqref{phi D} can be obtained by applying $\delta_{n+m}$ to the both sides of \eqref{psi D} and using formulas \eqref{phi delta}, \eqref{i delta}, \eqref{delta t}, and \eqref{delta D}.
 \end{proof}

 \subsection{The duality} Denote $d=\tau-1$. For any $c\in\C$, let $\mathcal{P}(c)$ be the unital subalgebra of $D(\C)$ generated by $u$ and $(u-c)d$. It is easy to check that any element of $\mathcal{P}(c)$ can be written in the form $\sum_{i=1}^{k}\sum_{j=1}^{l}a_{ij}u^{i}((u-c)d)^{j}$ for some $k,l\in\Z_{\geq 0}$, $a_{ij}\in\C$. It can be checked that there is an anti-isomorphism of algebras $\mathcal{F}_{n,m}: \mathcal{P}(n-1) \rightarrow \mathcal{P}(m-1)$ defined by
 \begin{align*}
     \mathcal{F}_{n,m}: \quad & (u-n+1)d  \mapsto (u+n), \\
      & u+m \mapsto (u-m+1)d.
 \end{align*}

 Let $ev_{n}: \Yn\rightarrow\Un$ be the evaluation homomorphism:
 \begin{equation}\label{ev}
 ev_{n}: t_{ij}(u)\mapsto \delta_{ij}+\frac{e_{ij}}{u}.
 \end{equation}
 For any $j\in\Z_{\geq 0}$, denote $u^{\overline{j}}=u(u+1)\dots (u+j-1)$ and $u^{\underline{j}}=u(u-1)\dots (u-j+1)$.
 Define the following difference operators:
 \begin{equation}\label{Dbar01}
 \overline{D}_{m}=(u+n)^{\underline{n+m}}\,ev_{n+m}\left[i_{n}(\qdet_{n}(u+n))\,\psi_{m}(D_{m})\right], 
 \end{equation}
 \begin{equation}\label{Dbar10}
     \overline{D}_{n}=(u+m)^{\underline{n+m}}\,ev_{n+m}\left[j_{m}(\qdet_{m}(-u-1))\,\phi_{n}(\widehat{\delta}_{n}D_{n})\right]. 
 \end{equation}
 Let us identify $\mathcal{F}_{n,m}$ with its extension $\id\otimes\mathcal{F}_{n,m}$ to $\Unm\otimes\mathcal{P}(n-1)$. 
The following theorem is the main result of this section:
 \begin{thm}\label{main1}
     The difference operators $\overline{D}_{n}$ and $\overline{D}_{m}$ belong to $\Unm\otimes \mathcal{P}(n-1)$ and $\Unm\otimes \mathcal{P}(m-1)$, respectively, and we have 
     \begin{equation}\label{mainrelation1}
     \mathcal{F}_{n,m}\bigl(\overline{D}_{n}\bigr)=\overline{D}_{m}.
     \end{equation}
 \end{thm}

\begin{proof}
Using formulas  \eqref{cdet}, \eqref{mDC}, \eqref{DC}, \eqref{ev}, we can write explicitly:
\begin{equation}\label{t1.1}
ev_{n+m}(D_{n+m}(C_{1,0}))=\sum_{\sigma\in S_{n+m}}(-1)^{\sigma}\,\overset{\longrightarrow}{\prod_{i=1}^{n}}\left(d\,\delta_{\sigma (i),i}-\frac{e_{\sigma (i),i}}{u-i+1}\right)\overset{\longrightarrow}{\prod_{i=n+1}^{n+m}}\left(-\delta_{\sigma (i),i}-\frac{e_{\sigma (i),i}}{u-i+1}\right),
\end{equation}
\begin{equation}\label{t1.2}
ev_{n+m}(D_{n+m}(C_{0,1}))=\sum_{\sigma\in S_{n+m}}(-1)^{\sigma}\,\overset{\longrightarrow}{\prod_{i=1}^{n}}\left(-\delta_{\sigma (i),i}-\frac{e_{\sigma (i),i}}{u-i+1}\right)\overset{\longrightarrow}{\prod_{i=n+1}^{n+m}}\left(d\,\delta_{\sigma (i),i}-\frac{e_{\sigma (i),i}}{u-i+1}\right),
\end{equation}
where arrows over products mean that we choose the following ordering: the index $i$ increases as we go from the left to the right.

For each subset $J=\{j_{1}<j_{2}<\dots <j_{k}\}\subset\{1,\dots, n+m\}$, introduce the following notation: $J_{1}=J\cap\{1,\dots ,n\}$,
$J_{2}=J\cap\{n+1,\dots ,n+m\}$, $\bar{J}=\{1,\dots, n+m\}\setminus J$, as well as
\[D^{(n)}_{J_{1}}=\frac{1}{u^{\underline{j_{1}-1}}}d\frac{1}{(u-j_{1})^{\underline{j_{2}-j_{1}-1}}}d\frac{1}{(u-j_{2})^{\underline{j_{3}-j_{2}-1}}}d\dots d\frac{1}{(u-j_{r})^{\underline{n-j_{r}}}},\]
where $r=|J_{1}|$, and 
\[D^{(m)}_{J_{2}}=\frac{1}{(u-n)^{\underline{j_{r+1}-n-1}}}d\frac{1}{(u-j_{r+1})^{\underline{j_{r+2}-j_{r+1}-1}}}d\frac{1}{(u-j_{r+2})^{\underline{j_{r+3}-j_{r+2}-1}}}d\dots d\frac{1}{(u-j_{k})^{\underline{n+m-j_{k}}}}.\]
Also, for each $J=\{j_{1}<j_{2}<\dots <j_{k}\}\subset\{1,\dots, n+m\}$ and $\sigma\in S_{n+m-|J|}$, define
\[e_{\sigma, J}=e_{j_{\sigma (1)},j_{1}}e_{j_{\sigma (2)},j_{2}}\dots e_{j_{\sigma (k)},j_{k}}.\]

From \eqref{t1.1} and \eqref{t1.2}, we get:
\begin{equation}\label{t1.3}
ev_{n+m}(D_{n+m}(C_{1,0}))=\sum_{J\subset\{1,\dots ,n+m\}}(-1)^{n+m-|J_{1}|}\sum_{\sigma\in S_{n+m-|J|}}(-1)^{\sigma}D^{(n)}_{J_{1}}\prod_{j\notin J_{2}}(u-j+1)^{-1}e_{\sigma, \bar{J}},
\end{equation}
\begin{equation}\label{t1.4}
ev_{n+m}(D_{n+m}(C_{0,1}))=\sum_{J\subset\{1,\dots ,n+m\}}(-1)^{n+m-|J_{2}|}\sum_{\sigma\in S_{n+m-|J|}}(-1)^{\sigma}\prod_{j\notin J_{1}}(u-j+1)^{-1}D^{(m)}_{J_{2}}e_{\sigma, \bar{J}}.
\end{equation}

Since for each $c\in\C$ and $j\in\Z_{\geq 0}$, we have
\[d\frac{1}{(u-c)^{\underline{j}}}=\frac{1}{(u+1-c)^{\underline{j}}}d-\frac{j}{(u+1-c)^{\underline{j+1}}},\]
the difference operator $D^{(m)}_{J_{2}}$ is of the form $\sum_{l=0}^{k-r}a_{l}((u-n)^{\underline{m-l}})^{-1}d^{l}$, where $a_{l}$ are some complex numbers. We have $(u-n)^{\underline{m}}((u-n)^{\underline{m-l}})^{-1}=(u-n-m+1)^{\overline{l}}$, and it is easy to show by induction on $l$ that $(u-n-m+1)^{\overline{l}}d^{l}=\sum_{i=0}^{l}\tilde{a}_{i}((u-n-m+1)d)^{i}$ for some complex numbers $\tilde{a}_{i}$. Therefore, $u^{\underline{m}}\tau^{n} D^{(m)}_{J_{2}} \tau^{-n}\in \mathcal{P}(m-1)$.
This, together with formulas \eqref{psi D} and \eqref{t1.4}, shows that $\overline{D}_{m}\in\Unm\otimes \mathcal{P}(m-1)$. Similarly, one can prove that $\overline{D_{n}}\in\Unm\otimes \mathcal{P}(n-1)$.

Comparing formulas \eqref{t1.3} and \eqref{t1.4}, we see that it is enough to prove that for any $J\subset\{1\lc n+m\}$, we have
\begin{equation}\label{t1.5}
    (u-n)^{\underline{m}}D^{(m)}_{J_{2}}=\prod_{j\in J_{2}}((u-n-m+1)d-(n+m-j))
\end{equation}
and 
\begin{equation}\label{t1.6}
    (u-m)^{\underline{n}}\,\widehat{\delta}_{n+m}D^{(n)}_{J_{1}}=(-1)^{n-|J_{1}|}\prod_{j\in J_{1}}((u-n-m+1)d-j+1).
\end{equation}

We will prove relation \eqref{t1.5}. Relation \eqref{t1.6} can be proved similarly. 

It is easy to check that both $D^{(m)}_{J_{2}}$ and $\prod_{j\in J_{2}}((u-n-m+1)d-(n+m-j))$ annihilate functions
\[f_{j}=(u-n-m+1)^{\overline{n+m-j}},\quad j\in J_{2}.\]

Denote $\tilde{D}=D^{(m)}_{J_{2}}-\prod_{j\in J_{2}}((u-n-m+1)d-(n+m-j))$. Then $\tilde{D}$ is a difference operator of the form $\sum_{i=0}^{s}a_{i}(u)((u-n-m-1)d)^{i}$. for some $s<|J_{2}|=\deg D^{(m)}_{J_{2}}$.

Suppose that $\tilde{D}\neq 0$. Then $\tilde{D}\cdot f_{j}=0$, $j\in J_{2}$ and $s<|J_{2}|$ imply that the difference Wronskian
$\det \bigl( (((u-n-m+1)d)^{p-1}\cdot f_{j_{s}})_{p,s=1\lc |J_{2}|}\bigr)$ is identically zero.

On the other hand, straightforward computation gives
\[\det \bigl( (((u-n-m+1)d)^{p-1}\cdot f_{j_{s}})_{p,s=1\lc |J_{2}|}\bigr)=\prod_{j\in J_{2}}(u-n-m+1)^{\overline{n+m-j}}\,\det \bigl( ((n+m-j_{s})^{p-1})_{p,s=1\lc |J_{2}|}\bigr).\]

Since $j_{s}\neq j_{p}$ for $s\neq p$, the Vandermonde determinant in the right hand side of the last formula is not zero. Therefore, $\tilde{D}=0$, and relation \eqref{t1.5} is proved.
\end{proof}
   
Let $\mathcal{Z}_{n+m}$ be the center of $\Unm$. For any subalgebra $\mathcal{A}$ of $\Unm$, we denote the subalgebra, generated by $\mathcal{A}$ and $\mathcal{Z}_{n+m}$ by $\mathcal{A}\cdot\mathcal{Z}_{n+m}$. We can now use Theorem \ref{main1} to establish the duality on the level of the universal enveloping algebra $\Unm$:
\begin{cor}\label{duality}
    The following equality holds:
    \begin{equation}\label{Equality in Unm}
    ev_{n+m}(\phi_{n}(\mathcal{B}_{n}))\cdot \mathcal{Z}_{n+m} =\,ev_{n+m}(\psi_{m}(\mathcal{B}_{m}))\cdot \mathcal{Z}_{n+m}.
    \end{equation}
\end{cor}
\begin{proof}
    Since by Theorem \ref{main1}, the difference operators $\overline{D}_{n}$ and $\overline{D}_{m}$ belong to $\Unm\otimes \mathcal{P}(n-1)$ and $\Unm\otimes \mathcal{P}(m-1)$, respectively, we can write
    \[\overline{D}_{n}=\sum_{i=0}^{M}\sum_{j=0}^{n}A^{(n)}_{ij}(u+m)^{i}((u-n+1)d)^{j},\quad \overline{D}_{m}=\sum_{i=0}^{m}\sum_{j=0}^{N} A^{(m)}_{ji}(u+n)^{j}((u-m+1)d)^{i}\]
    for some $M,N\in \Z_{\geq 0}$, $A^{(n)}_{ij}, A^{(m)}_{ji}\in \Unm$.

    Using \eqref{sylvester}, one can check that $i_{n}(\qdet_{n}(u+n))=qdet_{n+m}(u+n)\psi_{m}(\qdet^{-1}_{m}(u))$. Since the coefficients of $ev_{n+m}(qdet_{n+m}(u+n))$ belong to $\mathcal{Z}_{n+m}$ and the coefficients of $\qdet^{-1}_{m}(u)$ belong to $\mathcal{B}_{m}$, we deduce that the coefficients of $ev_{n+m}(i_{n}(\qdet_{n}(u+n)))$ belong to $ev_{n+m}(\psi_{m}(\mathcal{B}_{m}))\cdot \mathcal{Z}_{n+m}$. Therefore, it follows from the definition of $\overline{D}_{m}$ (see \eqref{Dbar01}) that $A^{(m)}_{ij}$ are elements of $ev_{n+m}(\psi_{m}(\mathcal{B}_{m}))\cdot \mathcal{Z}_{n+m}$. On the other hand. comparing coefficients for $d^{m}$ on the both sides of equation \eqref{Dbar01}, we see that 
    \[(u-m+1)^{m}\sum_{j=0}^{N} A^{(m)}_{jm}(u+n)^{j}=(u+n)^{\underline{n+m}}\,ev_{n+m}\left[i_{n}(\qdet_{n}(u+n))\right].\]
    Hence, coefficients of $ev_{n+m}\left[i_{n}(\qdet_{n}(u+n))\right]^{-1}$ can be expressed polynomially through $A^{(m)}_{ij}$. Therefore, using formula \eqref{Dbar01}, we can express all coefficients through $A^{(m)}_{ij}$, thus, the elements $\tilde{A}^{(m)}_{ji}$, $i=0\lc m$, $j=0\lc N$ generate the algebra $ev_{n+m}(\psi_{m}(\mathcal{B}_{m}))\cdot \mathcal{Z}_{n+m}$. Similarly, we prove that the elements $A^{(n)}_{ij}$, $i=0\lc M$, $j=0\lc n$ generate the algebra $ev_{n+m}(\phi_{n}(\mathcal{B}_{n}))\cdot\mathcal{Z}_{n+m}$. 
    But by relation \eqref{mainrelation1}, we can take $N=n$, $M=m$, and for each $i=0\lc n$, $j=0\lc m$, we have $A^{(n)}_{ij}= A^{(m)}_{ji}$.
\end{proof}

\begin{rem}
    It's not hard to prove Corollary \ref{duality} without Theorem \ref{main1} using the argument that we briefly outline now. We thank Vitaly Tarasov for pointing this out.
    
    For any $C=(c_{1}\lc c_{n+m})\in\C^{n+m}$, let $\widetilde{\mathcal{A}}_{n+m}(C)$ be the subalgebra of $\Unm$ generated by the coefficients of the difference operator $ev_{n+m}(D_{n+m}(C))$, that is $\widetilde{\mathcal{A}}_{n+m}(C)=ev_{n+m}(\mathcal{B}_{n+m}(C))$. It follows from relations \eqref{psi D}, \eqref{phi D}, \eqref{Dbar01}, \eqref{Dbar10}, and the proof of Corollary \ref{duality} that equality \eqref{Equality in Unm} is the same as
    \begin{equation}\label{duality tilde}
        \widetilde{\mathcal{A}}_{n+m}(C_{1,0})=\widetilde{\mathcal{A}}_{n+m}(C_{0,1}).
    \end{equation}
    It's not hard to see that for any non-zero complex number $a$, we have $\mathcal{B}_{n+m}(aC)=\mathcal{B}_{n+m}(C)$ (this just corresponds to changing $\tau$ to $a\tau$ in $D_{n+m}(C)$). Therefore, for any $a\in\C$, $a\neq 0$, we have $\widetilde{\mathcal{A}}_{n+m}(aC)=\widetilde{\mathcal{A}}_{n+m}(C)$.
    
    On the other hand, $\widetilde{\mathcal{A}}_{n+m}(C)$ can be obtained as the image under the evaluation map of the universal \textit{Gaudin} Bethe algebra $\mathfrak{B}_{n+m}(C)\in U(\glnm [t])$. For the construction of $\mathfrak{B}_{n+m}(C)$ with $C=(1\lc 1)$, see Section \ref{trigonometric Gaudin} below. To get $\mathfrak{B}_{n+m}(C)$ for any $C$, we just need to modify $\mathfrak{D}_{ij}$ as follows:  $\mathfrak{D}_{ij} = \delta_{ij}(\partial_{u}+c_{i})-e_{ij}(u)$. The equality $\widetilde{\mathcal{A}}_{n+m}(C)=\tilde{ev}_{n+m}(\mathfrak{B}_{n+m}(C))$ follows from the fact that the commutation relation for operators $\partial_{u}$, $u$ is the same as for the operators $\tau^{-1}$, $\tau u$. But it follows from the construction of $\mathfrak{B}_{n+m}(C)$ that it doesn't change if we add a vector of the form $(b\lc b)$ to $C$ (this just corresponds to changing $\partial_{u}$ to $\partial_{u}+b$ in differential operator $\mathfrak{D}_{n+m}(C)$). Therefore, $\widetilde{\mathcal{A}}_{n+m}(C+(b\lc b))=\widetilde{\mathcal{A}}_{n+m}(C)$.   

    Now, since $C_{1,0}=-C_{0,1}+(1\lc 1)$, we obtain relation \eqref{duality tilde}. 

    The invariance of $\widetilde{\mathcal{A}}_{n+m}(C)$ under transformation $C\mapsto aC+(b\lc b)$ can also be explained naturally in the context of quantum shift of argument algebra, see the next section for details. 
\end{rem}

\subsection{Shift of argument subalgebras} Let $\mathfrak{g}$ be any complex reductive Lie algebra. The following \emph{argument shift method} gives a way to construct
subalgebras in $S(\mathfrak{g})$ commutative with respect to the
Poisson--Lie bracket. Let
$ZS(\mathfrak{g})=S(\mathfrak{g})^{\mathfrak{g}}$ be the center of $S(\mathfrak{g})$ with respect to
the Poisson bracket, and let $\chi\in\mathfrak{g}^*$. Then the algebra $A(\chi)\subset S(\mathfrak{g})$ generated by the elements of $ZS(\mathfrak{g})\subset S(\mathfrak{g})=\mathbb{C}[\mathfrak{g}^*]$
shifted by $t\chi$ for all $t\in\mathbb{C}$ is commutative with respect
to the Poisson bracket, see \cite{FM}. The subalgebra $A(\chi)\subset S(\mathfrak{g})$ is called the \emph{shift of argument subalgebra} or \emph{Mishchenko-Fomenko subalgebra}. Identifying $\mathfrak{g}$ with $\mathfrak{g}^*$ with the help of an invariant non-degenerate inner product on $\mathfrak{g}$, we can regard $\chi$ as an element of $\mathfrak{g}$. Let $\mathfrak{z}(\chi)$ be the centralizer of $\chi$ in $\mathfrak{g}$. For semisimple  $\chi\in\mathfrak{g}$, the subalgebra $A(\chi)$ is known to be a maximal Poisson-commutative subalgebra of the maximal possible transcendence degree in the algebra $S(\mathfrak{g})^{\mathfrak{z}(\chi)}$ of $\mathfrak{z}(\chi)$-invariants in $S(\mathfrak{g})$. There is a natural lifting of $A(\chi)$ to a commutative subalgebra $\mathcal{A}(\chi)\subset U(\mathfrak{g})$, called the \emph{quantum} shift of argument subalgebra, see \cite{Ryb2} and \cite{FFTL} for details.

Let $\mathfrak{g}=\mathfrak{gl}_{n+m}$ and let $\chi\in\mathfrak{h}$ be a Cartan element of $\mathfrak{gl}_{n+m}$, i.e. a diagonal matrix. Then the subalgebra $\mathcal{A}(\chi)\subset U(\mathfrak{gl}_{n+m})$ is the only commutative subalgebra such that $A(\chi)=\gr\mathcal{A}(\chi)$ and $\mathcal{A}(\chi)\subset U(\mathfrak{gl}_{n+m})^{\mathfrak{z}(\chi)}$ (this follows from \cite[Corollary 10.12]{HKRW} and \cite{Tar2}). Since $A(\chi)\subset S(\mathfrak{gl}_{n+m})$ does not change under dilations of $\chi$ and under adding a scalar matrix to $\chi$, we have the following: 

\begin{prop}\label{pr:A-chi-aff-invariance}
    Denote $I=\sum_{i=1}^{n+m}e_{ii}$. Then for any $a,b\in\C$, $a\neq 0$, we have 
    \begin{equation}\label{A-chi-aff-invariance}
    \mathcal{A}(a\chi+b I)=\mathcal{A}(\chi).
    \end{equation}
\end{prop}

Consider the composition of the Olshanski homomorphism and the evaluation map $ev_{n+m}\circ\phi_n: Y(\mathfrak{gl}_n)\to U(\mathfrak{gl}_{n+m})$. Clearly, it takes Bethe subalgebras in the Yangian to some commutative subalgebras in $U(\mathfrak{gl}_{n+m})$. We have the following 

\begin{thm}\cite{IR}
   For any $C=(c_{1}\lc c_{n})$, we have $ev_{n+m}\circ\phi_{n} (\mathcal{B}_{n} (C))\subset\mathcal{A}(C,0)$, where $(C,0)$ is the block diagonal matrix with the upper-left $n\times n$ block $\operatorname{diag}(c_{1}\lc c_{n})$ and the lower-right $m\times m$ block $0$.
\end{thm}

In particular, $ev_{n+m}\circ\phi_{n} (\mathcal{B}_{n} (C))\subset\mathcal{A}(C_{1,0})$. Similarly, one can get $ev_{n+m}\circ\psi_{m} (\mathcal{B}_{m} (C))\subset\mathcal{A}(C_{0,1})$. Actually, we expect that $ev_{n+m}(\phi_{n}(\mathcal{B}_{n}))\cdot \mathcal{Z}_{n+m} = \mathcal{A}(C_{1,0})$ and $ev_{n+m}(\psi_{m}(\mathcal{B}_{m}))\cdot \mathcal{Z}_{n+m} = \mathcal{A}(C_{1,0})$. Then, since $C_{1,0} = -C_{1,0}+I$, Corollary \ref{duality} is just a special case of Proposition \ref{pr:A-chi-aff-invariance}. This observation was our first motivation to look at relation \eqref{A-chi-aff-invariance} in the context of the $(\gln, \glm)$-duality.

\section{The duality on the level of representations}\label{Sec4}
\subsection{} For any $\gl_{m}$-module $W$, let $W[b]$ denote the weight subspace of $W$ of weight $b$. Recall that a vector $w\in W$ is called singular if $e_{ij}w=0$ for any $i<j$. Denote by $W^{m\text{-sing}}[b]$ the subspace of singular vectors in $W[b]$.

Let $V$ be a finite-dimensional $\gl_{m}$-module, and let $M_{\mu}$ be the Verma $\gl_{m}$-module of highest weight $\mu$. Fix a vector $v_{\mu}\in M_{\mu}[\mu]$. Consider a linear map $\theta^{V}_{\mu,b}: (V\otimes M_{\mu})^{m\text{-sing}}[\mu+b]\rightarrow V[b]$ defined by  
\begin{equation}\label{theta}
v=\theta^{V}_{\mu,b}(v)\otimes v_{\mu}+\sum_{i}w_{i}\otimes v_{i},
\end{equation}
where $v_{i}\in M_{\mu}$ are weight vectors of weights distinct from $\mu$. The proof of the following proposition is a simple exercise.

\begin{prop}\label{theta'}
    There exists a Zariski open subset $I(V)$ of $\C^{m}$such that for any $\mu\in I(V)$ and any weight $b$ of the $\gl_{n}$-module $V$, the map $\theta^{V}_{\mu,b}$ is an isomorphism.
\end{prop}

Fix $a\in \Z_{\geq 0}^{n}$, $a=(a_{1},a_{2},\lc a_{n})$, and denote $S^{(m)}(a)=S^{a_{1}}\C^{m}\otimes S^{a_{2}}\C^{m}\otimes\dots\otimes S^{a_{n}}\C^{m}$. Also, fix a $\glm$-weight $\lambda^{(m)}\in I(S^{(m)}(a))$, a sequence of $n$ complex numbers $z=(z_{1}\lc z_{n})$, and
consider the following tensor product of evaluation $\Ym$-modules
\begin{equation}\label{S(a,z,lambda)}
S^{(m)}(a,z,\lambda^{(m)})=S^{a_{1}}\C^{m}(z_{1})\otimes S^{a_{2}}\C^{m}(z_{2})\otimes\dots\otimes S^{a_{n}}\C^{m}(z_{n})\otimes M_{\lambda^{(m)}}
\end{equation}
together with the corresponding homomorphism $\rho_{a,z,\lambda^{(m)}}:\Ym\rightarrow \End (S^{(m)}(a,z,\lambda^{(m)}))$.

Fix a $\glm$-weight $b$ of the representation $S^{(m)}(a)$, that is, $b=(b_{1},b_{2},\lc b_{m})\in\Z_{\geq 0}^{m}$, where $\sum_{j=1}^{m}b_{j}=\sum_{i=1}^{n}a_{i}$. It is known (see, for example, \cite[Proposition 4.7]{MTV6}) that the Bethe subalgebra $\mathcal{B}_{m}\subset\Ym$ commutes with the image of the inclusion $\Um\hookrightarrow\Ym$ given by $e_{ij}\mapsto t_{ij}^{(1)}$. Therefore, any element of $\rho_{a,z,\lambda^{(m)}}(\mathcal{B}_{m})$ sends $\bigl(S^{(m)}(a,z,\lambda^{(m)})\bigr)^{sing}[\lambda^{(m)}+b]$ to $\bigl(S^{(m)}(a,z,\lambda^{(m)})\bigr)^{sing}[\lambda^{(m)}+b]$. Therefore, using Proposition \ref{theta'}, we can define the following commutative subalgebra of $\End (S^{(m)}(a)[b])$:
\[\overline{\mathcal{B}}_{m}(a,z;b,\lambda^{(m)})=\{\theta^{S^{(m)}(a)}_{\lambda^{(m)},b}\circ l\circ (\theta^{S^{(m)}(a)}_{\lambda^{(m)},b})^{-1}\,\vert\, l\in\rho_{a,z,\lambda^{(m)}}(\mathcal{B}_{m})\}.\]

Now, consider the $\gl_{n}$-module $S^{(n)}(b)=S^{b_{1}}\C^{n}\otimes S^{b_{2}}\C^{n}\otimes\dots\otimes S^{b_{m}}\C^{n}$. Let $w=(w_{1}\lc w_{m})$ be a sequence of $m$ complex numbers and let $\lambda^{(n)}$ be a $\gln$-weight from $I(S^{(n)}(b))$. Consider the commutative subalgebra $\overline{\mathcal{B}}_{n}(b,w;a,\lambda^{(n)})$ of $\End (S^{(n)}(b)[a])$ obtained in the same way as $\overline{\mathcal{B}}_{m}(a,z;b,\lambda^{(m)})$ except that the roles of $m$ and $a$ are interchanged with the roles of $n$ and $b$, respectively, and the sequences $z$ and $\lambda^{(m)}$ are replaced with the sequences $w$ and $\lambda^{(n)}$, respectively. As we will see later, it is convenient to denote the components of $\lambda^{(n)}$ and $\lambda^{(m)}$ as follows: $\lambda^{(n)}=(-\lambda_{n},-\lambda_{n-1}\lc-\lambda_{1})$, $\lambda^{(m)}=(\lambda_{n+1}\lc\lambda_{n+m})$. Denote $\lambda=(\lambda_{1}\lc\lambda_{m+n})$. In what follows, for each $\lambda$ in some Zariski open subset of $\C^{m+n}$,  we will construct an isomorphism $\Psi_{\lambda}:S^{(m)}(a)[b]\rightarrow S^{(n)}(b)[a]$ such that under the corresponding identification of the algebras $\End (S^{(m)}(a)[b])$ and $\End (S^{(n)}(b)[a])$, the subalgebras $\overline{\mathcal{B}}_{m}(a,z;b,\lambda^{(m)})$ and $\overline{\mathcal{B}}_{n}(b,w;a,\lambda^{(n)})$ will coincide provided that 
\begin{equation}\label{z w}
z_{i}=-\lambda_{n-i+1}+a_{i}-i,\quad w_{j}=\lambda_{n+j}+b_{j}-j
\end{equation}
for each $i=1\lc n$, $j=1\lc m$.

\subsection{An isomorphism of Yangian modules}
Let $a$, $b$, $z$, $w$, $\lambda$, $\lambda^{(n)}$, and $\lambda^{(m)}$ be as in the previous section, in particular, relations \eqref{z w} hold. Define $\nu=(\nu_{1}\lc\nu_{n})\in\C^{n}$ and $\mu=(\mu_{1}\lc\mu_{m})\in\C^{m}$ by $\nu_{i}=\lambda^{(n)}_{i}-a_{n-i+1}$, $\mu_{j}=\lambda^{(m)}_{n+j}+b_{j}$, $i=1\lc n$, $j=1\lc m$. 

To distinguish between objects related to $\gln$, $\glm$, and $\glnm$, let us use superscripts $(n)$, $(m)$, and $(n+m)$, respectively. For example, $M_{\lambda^{(n)}}^{(n)}$, $M_{\lambda^{(m)}}^{(m)}$, and $M_{\lambda}^{(n+m)}$ are Verma modules of $\gln$, $\glm$, and $\glnm$, respectively.

Using the inclusion $i_{n}:\gln\hookrightarrow\glnm$, $i_{n}(e^{(n)}_{ij})= e^{(n+m)}_{ij}$, 
let us consider $M_{\lambda}^{(n+m)}$ as a $\gln$-module. 
Denote by  $M_{\lambda}^{n\text{-sing}}[\nu]$ the space of $\gln$-singular vectors of weight $\nu$ in $M_{\lambda}^{(n+m)}$.

Recall the Olshanski homomorphism $\psi_{m}: \Ym\hookrightarrow\Ynm$ introduced in Section \ref{Oh}. A remarkable property of $\psi_{m}$ is that $\psi_{m}(\Ym)$ commutes with $i_{n}(\Yn)$ (see \cite[Lemma 1.11.3 and Corollary 1.7.2]{Mol2}). In particular, $ev_{n+m}(\psi_{m}(\Ym))$ belongs to the centralizer $\Unm^{\gln}$ of $i_{n}(\gln)$ in $\Unm$. Since $\Unm^{\gln}$ acts on $M_{\lambda}^{n\text{-sing}}[\nu]$, the latter becomes a $\Ym$-module.

Recall the $\Ym$-module $S^{(m)}(a,z,\lambda^{(m)})$ defined by formula \eqref{S(a,z,lambda)}.

\begin{thm}\label{iso of Y m}
    For generic $\lambda$, the $\Ym$-module $M_{\lambda}^{n\text{-sing}}[\nu]$ is isomorphic to $S^{(m)}(a,z,\lambda^{(m)})$.
\end{thm}

Here and below, if $s\in\C^{k}$ for some $k$, then "for generic $s$" means "for $s$ from some Zariski open subset of $\C^{k}$. 

\begin{proof}
    If $\alpha$ is a complex number, denote $\alpha(u):=1+\alpha u^{-1}\in\C (u)$.
    The statement of the theorem follows from the fact that for generic $\lambda$, both $M_{\lambda}^{n\text{-sing}}[\nu]$ and $S^{(m)}(a,z,\lambda^{(m)})$ are irreducible highest-weight $\Ym$-modules of the same highest $\Ym$-weight
    \begin{equation}\label{weight1}
\bigl(\lambda_{n+1}(u)\prod_{i=1}^{n}a_{i}(u-z_{i}),\, \lambda_{n+2}(u),\,\lambda_{n+3}(u)\lc\lambda_{n+m}(u)\bigr).
\end{equation}

For $M_{\lambda}^{n\text{-sing}}[\nu]$, this can be proved using the same arguments as in the proof of \cite[Theorem 8.5.4]{Mol2}, we just need to replace the finite-dimensional irreducible representation of $\glnm$ with the Verma module $M_{\lambda}^{(n+m)}$. In $S^{(m)}(a,z,\lambda^{(m)})$ there is an obvious singular vector of weight \eqref{weight1}, which is the tensor profuct of highest-weight vectors in each tensor factor of $S^{(m)}(a,z,\lambda^{(m)})$. The irreducibility of $S^{(m)}(a,z,\lambda^{(m)})$ for generic $\lambda$ can be proved using the arguments, which are similar to those used in the proof of \cite[Theorem 6.1.1]{Mol2}. The only essential difference is that we have to replace dual modules with \textit{restricted} dual modules. (If $V$ is a $\Ym$-module with $\glm$-weight decomposition $V=\bigoplus_{\eta}V[\eta]$ such that $\dim V[\eta]<\infty$ for any $\eta$, the restricted dual $V^{\dagger}$ is the following submodule of the dual module $V^{*}$: $V^{\dagger}=\bigoplus_{\eta}(V[\eta])^{*}$.)   
\end{proof}

We will now deduce the $\Yn$-analog of Theorem \ref{iso of Y m}. Using the inclusion $j_{m}:\glm\hookrightarrow\glnm$, $j_{m}(e^{(m)}_{ij})= e^{(n+m)}_{n+i,n+j}$, let us consider $M_{\lambda}^{(n+m)}$ as a $\glm$-module. Denote by $M_{\lambda}^{m\text{-sing}}[\mu]$ the space of $\glm$-singular vectors of weight $\mu$ in $M_{\lambda}^{(n+m)}$. 

Instead of the map $\psi_{m}$, we will need the map $\phi_{n}$, also introduced in Section \ref{Oh}. Since $\phi_{n}=\delta_{n+m}\psi_{n}\delta_{n}$, the homomorphism $\phi_{n}$ has similar properties to $\psi_{m}$, in particular, $ev_{n+m}(\phi_{n}(\Yn))$ belongs to the centralizer $\Unm^{\glm}$ of $j_{m}(\glm)$ in $\Unm$. Since $\Unm^{\glm}$ acts on $M_{\lambda}^{m\text{-sing}}[\mu]$, the latter becomes a $\Yn$-module.

For any sequence $\eta=(\eta_{1},\eta_{2}\lc\eta_{k})\in\C^{k}$, denote $\eta_{*}=(-\eta_{k},-\eta_{k-1}\lc -\eta_{1})$.
If $V$ is a representation of the Yangian $\Yn$ with the corresponding homomorphism $\rho: \Yn\rightarrow\End (V)$, write $V^{\vee}$ for the $\Yn$-module such that $V^{\vee}=V$ as vector spaces, and the corresponding homomorphism $\rho^{\vee}:\Yn\rightarrow\End (V^{\vee})$ is given by $\rho^{\vee}=\rho\circ\delta_{n}$. Notice that $V^{n\text{-sing}}[\eta]=(V^{\vee})^{n\text{-sing}}[\eta_{*}]$ for any $\gln$-weight $\eta$ of $V$. 

As $\Ynm$-modules, $(M_{\lambda}^{(n+m)})^{\vee}\cong M_{\lambda_{*}}^{(n+m)}$. Together with the relation $\phi_{n}\delta_{n}=\delta_{n+m}\psi_{n}$, this gives the isomorphism of $\Yn$-modules $(M_{\lambda}^{m\text{-sing}}[\mu])^{\vee}\cong M_{\lambda_{*}}^{m\text{-sing}}[\mu_{*}]$, where we consider $M_{\lambda_{*}}$ as a $\glm$-module via the inclusion $i_{m}:\glm\hookrightarrow\glnm$, $i_{m}(e^{(m)}_{ij})= e^{(n+m)}_{ij}$, and $\Yn$ acts on $M_{\lambda_{*}}^{m\text{-sing}}[\mu_{*}]$ via the homomorphism $\Yn\xrightarrow{ev_{n+m}\circ\psi_{n}} U(\glnm)^{\glm}$. Therefore, interchanging $n$ and $m$ and replacing $\lambda$ with $\lambda_{*}$, $\mu$ with $\mu_{*}$, and $\nu$ with $\nu_{*}$ in Theorem \ref{iso of Y m}, we obtain: 

\begin{thm}\label{iso of Y n}
For generic $\lambda$, the $\Yn$-module $(M_{\lambda}^{m\text{-sing}}[\mu])^{\vee}$ is isomorphic to $S^{(n)}(b,w,\lambda^{(n)})$.
\end{thm}

\subsection{The duality}
Let $M_{\lambda}^{sing}[\nu,\mu]$ denote the subspace of vectors of $M^{(n+m)}_{\lambda}$, which are $\gln$-singular of weight $\nu$ and $\glm$-singular of weight $\mu$. Notice that 
\[M_{\lambda}^{sing}[\nu,\mu]=\bigl(M_{\lambda}^{n\text{-sing}}[\nu]\bigr)^{m\text{-sing}}[\mu]=\bigl(\bigl(M_{\lambda}^{m\text{-sing}}[\mu]\bigr)^{\vee}\bigr)^{n\text{-sing}}[\nu_{*}].\]
Therefore, for generic $\lambda$, Theorems \ref{iso of Y m} and \ref{iso of Y n} provide two vector space isomorphisms: 
\[\zeta^{(m)}: M_{\lambda}^{sing}[\nu,\mu]\rightarrow (S^{(m)}(a,z,\lambda^{(m)}))^{m\text{-sing}}[\mu]\]
and 
\[\zeta^{(n)}: M_{\lambda}^{sing}[\nu,\mu]\rightarrow (S^{(n)}(b,w,\lambda^{(n)}))^{n\text{-sing}}[\nu_{*}],\] 
respectively.

Let $J_{a,b}$ be the Zariski open subset of all values of $\lambda$ such that all maps in the chain below are well-defined isomorphisms:
\begin{equation}\label{chain}
\begin{split}
 S^{(m)}(a)[b] & \xlongrightarrow{\left(\theta^{S^{(m)}(a)}_{\lambda^{(m)},b}\right)^{-1}} (S^{(m)}(a,z,\lambda^{(m)}))^{m\text{-sing}}[\mu]\xlongrightarrow{\left(\zeta^{(m)}\right)^{-1}} M_{\lambda}^{sing}[\nu,\mu]\xlongrightarrow{\zeta^{(n)}}\\
 & \xlongrightarrow{\zeta^{(n)}} (S^{(n)}(b,w,\lambda^{(n)}))^{n\text{-sing}}[\nu_{*}]\xlongrightarrow{\theta^{S^{(n)}(b)}_{\lambda^{(n)},a}}S^{(n)}(b)[a].
\end{split}
\end{equation}

If $l:V\rightarrow W$ is an isomorphism of vector spaces, let $\widehat{l}: \End (V)\rightarrow\End (W)$ denote the isomorphism of algebras such that $\widehat{l}(A)=l\circ A\circ l^{-1}$. Assume that $\lambda\in J_{a,b}$. Then the chain of isomorphisms \eqref{chain} defines a linear isomorphism $\Psi_{\lambda}: S^{(m)}(a)[b]\rightarrow S^{(n)}(b)[a]$, which, in turn, gives an isomorphism $\widehat{\Psi}_{\lambda}:\End (S^{(m)}(a)[b])\rightarrow \End (S^{(n)}(b)[a])$. 

Recall the algebras $\overline{\mathcal{B}}_{m}(a,z;b,\lambda^{(m)})\in\End (S^{(m)}(a)[b])$ and $\overline{\mathcal{B}}_{n}(b,w;a,\lambda^{(n)})\in\End (S^{(n)}(b)[a])$.
By construction, these algebras are generated by the coefficients of the following difference operators: 
\[D_{m}(a,z;b,\lambda^{(m)})=\widehat{\theta^{S^{(m)}(a)}_{\lambda^{(m)},b}}\bigl(\rho_{a,z,\lambda^{(m)}}(D_{m})\bigr)\] 
and 
\[D_{n}(b,w;a,\lambda^{(n)})=\widehat{\theta^{S^{(n)}(b)}_{\lambda^{(n)},a}}\bigl(\rho_{b,w,\lambda^{(n)}}(D_{n})\bigr),\]
respectively

\begin{thm}\label{duality rep} Suppose that $\lambda\in J_{a,b}$. Then there exist generators $A^{(m)}_{ij}$, $i=1\lc n$, $j=1\lc m$, of $\overline{\mathcal{B}}_{m}(a,z;b,\lambda^{(m)})$ and generators $A^{(n)}_{ji}$, $i=1\lc n$, $j=1\lc m$, of $\overline{\mathcal{B}}_{n}(b,w;a,\lambda^{(n)})$ such that
\[u^{\underline{m}}\prod_{i=1}^{n}(u-z_{i})\,\,D_{m}(a,z;b,\lambda^{(m)})=\sum_{i=0}^{n}\sum_{j=0}^{m}A^{(m)}_{ij}(u+n)^{i}((u-m+1)d)^{j},\]
\[u^{\underline{n}}\prod_{i=1}^{m}(u-w_{i})\,\,D_{n}(b,w;a,\lambda^{(n)})=\sum_{i=0}^{n}\sum_{j=0}^{m}A^{(n)}_{ji}(u+m)^{j}((u-n+1)d)^{i},\]
and
\begin{equation}\label{duality explicit}
\widehat{\Psi}_{\lambda}(A^{(m)}_{ij})=A^{(n)}_{ji}.
\end{equation}
\end{thm}
\begin{proof}
    Using that $i_{n}(\qdet_{n}(u))=t^{(n+m)}_{\{1\lc n\}}(u)$, where $t^{(n+m)}_{\{1\lc n\}}(u)$ is given by the formula \eqref{qminor}, and that $t^{(n+m)}_{ii}(u)v=\nu_{i}(u)v$, $t^{(n+m)}_{kl}(u)v=0$ for all $v\in M_{\lambda}^{sing}[\mu,\nu]$, $i=1\lc n$, $1\leq k < l\leq n$, it is easy to check that $i_{n}(\qdet_{n}(u+n))$ acts on $M_{\lambda}^{sing}[\mu,\nu]$ as $\id \cdot\prod_{i=1}^{n}\left(1+\nu_{i}(u+n-i+1)^{-1}\right)$. Using the relation $z_{i}=-\nu_{n-i+1}-i$, we get that $(u+m)^{\underline{m+n}}\, i_{n}(\qdet_{n}(u+n))$ acts on $M_{\lambda}^{sing}[\mu,\nu]$ as $\id \cdot u^{\underline{n}}\,\prod_{i=1}^{n}(u-z_{i})$. Similarly, $(u+n)^{\underline{m+n}}\, j_{m}(\qdet_{m}(-u-1))$ acts on $M_{\lambda}^{sing}[\mu,\nu]$ as $\id \cdot u^{\underline{m}}\,\prod_{i=1}^{m}(u-w_{i})$. Therefore, using that $\zeta^{(m)}$ and $\zeta^{(n)}$ are restrictions of homomorphisms of Yangian modules, we deduce
    \[u^{\underline{m}}\prod_{i=1}^{n}(u-z_{i})\,D_{m}(a,z;b,\lambda^{(m)})=\widehat{\theta^{S^{(m)}(a)}_{\lambda^{(m)},b}}(\widehat{\zeta^{(m)}}(\rho_{\lambda}(\overline{D}_{m}))),\]
    \[u^{\underline{n}}\prod_{i=1}^{m}(u-w_{i})\,D_{n}(b,w;a,\lambda^{(n)})=\widehat{\theta^{S^{(n)}(b)}_{\lambda^{(n)},a}}(\widehat{\zeta^{(n)}}(\rho_{\lambda}(\overline{D}_{n}))),\]
    where $\overline{D}_{m}$ and $\overline{D}_{n}$ are defined in formulas \eqref{Dbar01} and \eqref{Dbar10}, respectively, and $\rho_{\lambda}: \Unm\rightarrow \End (M^{(n+m)}_{\lambda})$ is the homomorphism, corresponding to the action of $\Unm$ on the Verma module $M^{(n+m)}_{\lambda}$. Now Theorem \ref{duality rep} follows from Theorem \ref{main1}.

\end{proof}

\section{Degeneration to the duality between the XXX-type spin chain and the trigonometric Gaudin model}\label{degeneration}
Let $a$, $b$, $z$, $w$, $\lambda$, $\lambda^{(n)}$, and $\lambda^{(m)}$ be as introduced in the previous section. 

\subsection{The Bethe algebra of the quantum XXX-type spin chain}\label{5.1} Consider the following tensor product of evaluation $\Ym$-modules
\begin{equation}\label{S(a,z)} 
S^{(m)}(a,z)=S^{a_{1}}\C^{m}(z_{1})\otimes S^{a_{2}}\C^{m}(z_{2})\otimes\dots\otimes S^{a_{n}}\C^{m}(z_{n}). 
\end{equation}

Recall that in Section \ref{Bethe subalgebras in Yangian}, for a sequence of complex numbers $C=(c_{1}\lc c_{m})$, we introduced the Bethe algebra $\mathcal{B}_{m}(C)\subset \Ym$ together with its generating difference operator $D_{m}(C)$. As a subalgebra of $\Ym$, the Bethe algebra $\mathcal{B}_{m}(C)$ acts on $S^{(m)}(a,z)$. It is known (see e.g. \cite[Proposition 4.7]{MTV6}) that $\mathcal{B}_{m}(C)$ commutes with elements $t_{ii}^{(1)}$, $i=1\lc m$, therefore, $\mathcal{B}_{m}(C)$ acts on $S^{(m)}(a)[b]\subset S^{(m)}(a,z)$. Let $\rho_{a,z}:\mathcal{B}_{m}(C)\rightarrow \End (S^{(m)}(a)[b])$ be the corresponding homomorphism. 

Suppose that $C=(\xi_{1}^{-1}\lc\xi_{m}^{-1})$ for some nonzero $\xi_{i}$, $i=1\lc m$. Denote $\xi=(\xi_{1}\lc\xi_{m})$ and introduce the difference operator $D_{m}^{XXX}(a,\lambda^{(n)};b,\xi)=\left(\prod_{i=1}^{m}\xi_{i}\right)\rho_{a,z}(D_{m}(C))$. The dependence on $\lambda^{(n)}$ indicated in the last formula is through relation between $z$ and $\lambda^{(n)}$ given in \eqref{z w}.

By construction, the difference operator $D_{m}^{XXX}(a,\lambda^{(n)};b,\xi)$ is of the form

\[D_{m}^{XXX}(a,\lambda^{(n)};b,\xi)=\sum_{i=0}^{m}\left(\sum_{j=0}^{\infty}B^{XXX}_{ij}u^{-j}\right)\tau^{m-i}\]
for some commuting linear operators $B^{XXX}_{ij}\in\End (S^{(m)}(a)[b])$, which depend polynomially on $\lambda^{(n)}$ and $\xi$. We, therefore, define $B^{XXX}_{ij}$ for all $\lambda^{(n)}$ and $\xi$ by continuity, and denote the algebra generated by $B^{XXX}_{ij}$, $i=1\lc n$, $j\in \Z_{\geq 0}$ by $\mathcal{B}_{m}^{XXX}(a,\lambda^{(n)};b,\xi)$. The algebra $\mathcal{B}_{m}^{XXX}(a,\lambda^{(n)};b,\xi)$ is known as the algebra of higher Hamiltonians (the Bethe algebra) of the (quasi-periodic) spin chain of XXX-type (see \cite{KS}, \cite{MTV6}).

\subsection{The Bethe algebra of the quantum trigonometric Gaudin model}\label{trigonometric Gaudin}
In Section \ref{Bethe subalgebras in Yangian}, for any associative algebra $\mathcal{A}$, we introduced the algebra $D(\mathcal{A})$ of difference operators with coefficients in $\mathcal{A}((u^{-1}))$. We will now consider the algebra $\mathfrak{D}(\mathcal{A})$ of \textit{differential} operators with coefficients in $\mathcal{A}((u^{-1}))$. That is, an element of $\mathfrak{D}(\mathcal{A})$ is a differential operator of the form $\sum_{i=0}^{l} f_{i}(u)\partial_{u}^{i}$ for some $l$, where $f_{i}(u)\in\mathcal{A}((u^{-1}))$, $i=1\lc l$ and $\partial_{u}=\frac{d}{du}$.

Consider the current Lie algebra $\gln [t] = \gln \otimes \C [t]$ with the Lie bracket given by
\[ [g\otimes t^{i}, h\otimes t^{j}]= [g,h]\otimes t^{i+j}, \quad g,h\in\gln .  \]
For any $i,j=1\lc n$, introduce the elements $e_{ij}(u)\in U(\gln [t])((u^{-1})$:
\[e_{ij}(u)=\sum_{s=0}^{\infty}(e_{ij}\otimes t^{s}) u^{-s-1}.\]

Define
$\mathfrak{D}_{n}\in \mathfrak{D}(U(\gln [t]))$ as follows:

\begin{equation}\label{mathfrak D}
\mathfrak{D}_{n}= \sum_{\sigma\in S_{n}}(-1)^{\sigma}\mathfrak{D}_{\sigma (1),1}\mathfrak{D}_{\sigma (2), 2}\dots \mathfrak{D}_{\sigma (n), n},
\end{equation}
where
\[
\mathfrak{D}_{ij} = \delta_{ij}\partial_{u}-e_{ij}(u).
\]
Let $\mathfrak{B}_{ij}\in U(\gln [t])$, $i=0\lc n$, $j\in\Z_{\geq 0}$ be the coefficients of $\mathfrak{D}_{n}$:
\begin{equation}
\mathfrak{D}_{n}=\sum_{i=0}^{n}\left(\sum_{j=0}^{\infty}\mathfrak{B}_{ij}u^{-j}\right)\partial_{u}^{n-i}.
\end{equation}

Define the universal Gaudin Bethe algebra $\mathfrak{B}_{n}$ to be the unital subalgebra of $U(\gln [t])$ generated by $\mathfrak{B}_{ij}\in U(\gln [t])$, $i=0\lc n$, $j\in\Z_{\geq 0}$. It is known, see \cite{Tal},\cite{MTV6}, that $\mathfrak{B}_{n}$ is commutative and belongs to the centralizer of $\gln\otimes 1\in U(\gln [t])$. 

We have homomorphisms $\widetilde{ev}_{n}: U(\gln [t])\rightarrow \Un$ and $\widetilde{sh}_{w}: U(\gln [t])\rightarrow U(\gln [t])$, where $w\in\C$, given by $\widetilde{ev}_{n}(e_{ij}(u))=e_{ij}\cdot u^{-1}$, $\widetilde{sh}_{w}(e_{ij}(u))=e_{ij}(u-w)$. If $V$ is a $\Un$-module, let $V(w)$ be the $U(\gln [t])$-module obtained by pulling back $V$ via $\widetilde{sh}_{w}\circ\widetilde{ev}_{n}$. The module $V(w)$ is called an evaluation module with evaluation parameter $w$. 

Consider the following tensor product of evaluation $U(\gln [t])$-modules:
\begin{equation}\label{S_U}
S_{U}^{(n)}(b,\xi,\lambda^{(n)})=S^{b_{1}}\C^{n}(\xi_{1})\otimes S^{b_{2}}\C^{n}(\xi_{2})\otimes\dots\otimes S^{b_{m}}\C^{n}(\xi_{m})\otimes M^{(n)}_{\lambda^{(n)}}(0). 
\end{equation}
As a subalgebra of $U(\gln [t])$, the universal Gaudin Bethe algebra $\mathfrak{B}_{n}$ acts on $S_{U}^{(n)}(b,\xi,\lambda^{(n)})$. Since $\mathfrak{B}_{n}$ commutes with $\gln\otimes 1\in U(\gln [t])$, it acts on $(S^{(n)}(b)\otimes M^{(n)}_{\lambda^{(n)}})^{sing}[\lambda^{(n)}+a]\subset S_{U}^{(n)}(b,\xi,\lambda^{(n)})$. Let $\rho^{G}_{b,\,\xi,\,\lambda^{(n)}}:\mathfrak{B}_{n}\rightarrow \End ((S^{(n)}(b)\otimes M^{(n)}_{\lambda^{(n)}})^{sing}[\lambda^{(n)}+a])$ be the corresponding homomorphism. 

Introduce the differential operator \[\mathfrak{D}_{n}^{tG}(b,\xi;a,\lambda^{(n)})=\widehat{\theta^{S^{(n)}(b)}_{\lambda^{(n)},a}}\left(\rho^{G}_{b,\,\xi,\,\lambda^{(n)}}(\mathfrak{D}_{n})\right)\in\mathfrak{D}(\End (S^{(n)}(b)[a])).\] 
By construction, the differential operator $\mathfrak{D}_{n}^{tG}(b,\xi;a,\lambda^{(n)})$ is of the form

\[\mathfrak{D}_{n}^{tG}(b,\xi;a,\lambda^{(n)})=\sum_{i=0}^{n}\left(\sum_{j=0}^{\infty}B^{tG}_{ij}u^{-j}\right)\partial_{u}^{n-i}\]
for some commuting linear operators $B^{tG}_{ij}\in\End (S^{(n)}(b)[a])$. The subalgebra $\mathcal{B}_{n}^{tG}(b,\xi;a,\lambda^{(n)})$ of $\End (S^{(n)}(b)[a])$ generated by $B^{tG}_{ij}$, $i=1\lc n$, $j\in \Z_{\geq 0}$ is called the Bethe algebra of the trigonometric Gaudin model, see \cite{IKR}, \cite{MR}.

\subsection{The duality}
Let $\{e^{(n)}_{i},\, i=1\lc n\}$ be the standard basis of $\C^{n}$, that is, $e_{i}^{(n)}=(0\lc 0,1,0\lc 0)$, where $1$ is on the $i$-th place. Let $\{e^{(m)}_{i},\, i=1\lc m\}$ be the similar basis of $\C^{m}$. For any $n\times m$ matrix $K=(k_{ij})_{\substack{i=1\lc n \\ j=1\lc m}}$ with non-negative integer entries, define vectors $v^{(n)}_{K}\in (S\C^{(n)})^{\otimes m}$ and $v^{(m)}_{K}\in (S\C^{(m)})^{\otimes n}$ as follows:
\[v_{K}^{(n)}=\prod_{i=1}^{n}(e_{i}^{(n)})^{k_{i1}}\otimes \prod_{i=1}^{n}(e_{i}^{(n)})^{k_{i2}}\otimes\dots\otimes \prod_{i=1}^{n}(e_{i}^{(n)})^{k_{im}},\]
\[v_{K}^{(m)}=\prod_{i=1}^{m}(e_{i}^{(m)})^{k_{1i}}\otimes \prod_{i=1}^{m}(e_{i}^{(m)})^{k_{2i}}\otimes\dots\otimes \prod_{i=1}^{m}(e_{i}^{(m)})^{k_{ni}}.\]

Consider the following linear isomorphism
\begin{equation}\label{Psi}
    \begin{split}
        \Psi: (S\C^{(m)})^{\otimes n} & \rightarrow (S\C^{(m)})^{\otimes n}, \\
        v_{K}^{(m)} & \mapsto v_{K}^{(n)}.
    \end{split}
\end{equation}
Notice that 
\[S^{(m)}(a)[b]=\{v_{K}^{(m)}\,\vert\, K=(k_{ij})_{\substack{i=1\lc n \\ j=1\lc m}},\, \sum_{j}k_{ij}=a_{i},\, \sum_{i}k_{ij}=b_{j}\},\]
\[S^{(n)}(b)[a]=\{v_{K}^{(n)}\,\vert\, K=(k_{ij})_{\substack{i=1\lc n \\ j=1\lc m}},\, \sum_{j}k_{ij}=a_{i},\, \sum_{i}k_{ij}=b_{j}\}.\]
Therefore, $\Psi (S^{(m)}(a)[b])=S^{(n)}(b)[a]$, and we have an isomorphism $\widehat{\Psi}: \End (S^{(m)}(a)[b]) \rightarrow \End (S^{(n)}(b)[a])$. The next theorem is the main result of Section \ref{degeneration}.

\begin{thm}\label{main3}
Denote $\boldsymbol{n}=(n,n\lc n)\in\C^{n}$. For any $\lambda^{(n)}\in\C^{n}$ and $\xi\in\C^{m}$, there exists generators $\widetilde{A}^{(m)}_{ij}$, $i=1\lc n$, $j=1\lc m$, of $\mathcal{B}_{m}^{XXX}(a,\lambda^{(n)}-\boldsymbol{n};b,\xi)\subset \End (S^{(m)}(a)[b])$ and generators $\widetilde{A}^{(n)}_{ji}$, $i=1\lc n$, $j=1\lc m$, of $\mathcal{B}_{n}^{tG}(b,\xi;a,\lambda^{(n)})$ such that
\begin{equation}\label{coef A tilde m}
\prod_{i=1}^{n}(u-z_{i}-n)\,\,D_{m}^{XXX}(a,\lambda^{(n)}-\boldsymbol{n};b,\xi)=\sum_{i=0}^{n}\sum_{j=0}^{m}\widetilde{A}^{(m)}_{ij}u^{i}\tau^{j},
\end{equation}
\begin{equation}\label{coef A tilde n}
u^{n}\prod_{i=1}^{m}(u-\xi_{j})\,\,\mathfrak{D}_{n}^{tG}(b,\xi;a,\lambda^{(n)})=\sum_{i=0}^{n}\sum_{j=0}^{m}\widetilde{A}^{(n)}_{ji}u^{j}(u\partial_{u})^{i},
\end{equation}
and for each $i=0\lc n$, $j=0\lc m$, we have
\[\widehat{\Psi}(\widetilde{A}^{(m)}_{ij})=\widetilde{A}^{(n)}_{ji}.\]
\end{thm}

The rest of the Section \ref{degeneration} will be devoted to the proof of Theorem \ref{main3}. The main idea of the proof is to consider Theorem \ref{duality rep} with $\lambda^{(m)}=r\xi$, $r\in\C$, in the limit $r\to\infty$.

\subsection{Filtration and grading}
Suppose that $I$ is a totally ordered set. Let $V$ be an $I$-filtered vector space, that is $V=\sum_{i\in I}\mathcal{F}_{i}$, where $\mathcal{F}_{i}$, $i\in I$, are vector subspaces of $V$ called filtration components such that for any $i,j\in I$, $i\leq j$, we have $\mathcal{F}_{i}\subset \mathcal{F}_{j}$. Then we denote by $gr \,V$ the associated graded vector space of $V$, that is, $gr\, V=\bigoplus_{i\in I}\bigl(\mathcal{F}_{i}/\bigl(\sum_{j<i}\mathcal{F}_{j}\bigr)\bigr)$. For any $i\in I$, let $gr_{i}: \mathcal{F}_{i}\rightarrow \mathcal{F}_{i}/\bigl(\sum_{j<i}\mathcal{F}_{j}\bigr)$ be the canonical projection. If $v\in V$, define the degree $\deg (v)$ of $v$ as the element $i\in I$ such that $v\in\mathcal{F}_{i}$ and $v\notin\mathcal{F}_{j}$ for $j<i$. Denote $gr\, v = gr_{\deg (v)} v$.

Now, suppose that $I$ is a totally ordered abelian group, and let $\mathcal{A}$ be an associative algebra. We say that $\mathcal{A}$ is an $I$-filtered algebra (or on $\mathcal{A}$ we have a filtration) if $\mathcal{A}$ is a $\Z$-filtered vector space with filtration components $\mathcal{F}_{i}$, $i\in I$ such that for any $f\in\mathcal{F}_{i}$, $h\in\mathcal{F}_{j}$, we have $fh\in \mathcal{F}_{i+j}$. If $\mathcal{A}$ is an $I$-filtered algebra, then $gr\, \mathcal{A}$ is naturally an associative algebra (called the associated graded algebra of $\mathcal{A}$) with multiplication given as follows: for any $i,j\in I$, $f\in\mathcal{F}_{i}$, $h\in\mathcal{F}_{j}$, we have $gr_{i}(f)\cdot gr_{j}(h)= gr_{i+j}(fh)$.

Almost everywhere below, we will have $I=\Z$. The only exception is Subsection \ref{4.7.2}, where we will need $I=\mathbb{Q}$. 
\subsection{Degeneration to trigonometric Gaudin model}\label{5.5}
Recall the evaluation homomorphism $ev_{n,c}:\Yn\rightarrow \Un$, where $c\in\C$:
\[ev_{n,c}: t_{ij}(u)\mapsto\delta_{ij}+\frac{e_{ij}}{u-c}.\]
We can think of $c$ as variable, then $\frac{1}{u-c}$ becomes an element of $(\C[c])((u^{-1}))$, and $ev_{n,c}$ becomes a map  
$\Yn\rightarrow \Un\otimes\C[c]$.

Therefore, in the construction of the algebra $\overline{\mathcal{B}}_{n}(b,w;a,\lambda^{(n)})$, we can consider $\lambda^{(m)}=(\lambda_{n+1}\lc\lambda_{n+m})$ as a set of algebraically independent commuting variables. Namely, denoting the field of of rational functions of $\lambda^{(m)}$ as $\C(\lambda^{(m)})$, let us think of $\overline{\mathcal{B}}_{n}(b,w;a,\lambda^{(n)})$ as a subalgebra of $\End(S^{(n)}(b)[a])\otimes\C(\lambda^{(m)})$ generated by the image of the Bethe algebra $\mathcal{B}_{n}\subset\Yn$ and $\C(\lambda^{(m)})$. 

In a similar way, we can consider the parameters $\xi=(\xi_{1}\lc\xi_{m})$ introduced in Section \ref{5.1} as variables. Let $\C(r, \xi)$ be the algebra of rational functions in variables $\xi_{1}\lc\xi_{m}$ and $r$. Any polynomial in $r$ and $\xi$ is invertible in $\C(\xi)((r^{-1}))$, so we can consider $\C(r, \xi)$ as a subalgebra in $\C(\xi)((r^{-1}))$. If $f=\sum_{i=0}^{\infty}f_{i}r^{k-i}\in\C(\xi)((r^{-1}))$, define the degree $\deg (f)$ of $f$ by $\deg (f) = k$. Consider the homomorphism of algebras $\pi_{r,\xi}:\C(\lambda^{(m)})\rightarrow\C(r, \xi)$ given by $\pi_{r,\xi}(\lambda_{n+i})=r\xi_{i}$, $i=1\lc n$. Then for any $p\in\C(\lambda^{(m)})$, define the degree $\deg (p)$ of $p$ by $\deg (p)=\deg (\pi_{r,\xi}(p))$. For any vector space $V$, introduce a filtration on $V\otimes\C(\lambda^{(m)})$ by defining the filtration components $\mathcal{F}_{i}$, $i\in\Z$ as follows:
\begin{equation}\label{filtr lambda}
    \mathcal{F}_{i}= span_{\C}\{v\otimes p\,\vert\,\deg (p)\leq i\}.
\end{equation}
If $f=\sum_{i}v_{i}\otimes p_{i}\in V\otimes\C(\lambda^{(m)})$ is of degree $k$, and we have $\pi_{r,\xi} (p_{i})=\sum_{j=0}^{\infty}p_{ij}(\xi)r^{k-j}$, let us identify $gr\, (f)$ and $\sum_{i}v_{i}\otimes p_{i0}(\xi)$, that is, we look at  $gr\, (V\otimes \C(\lambda^{(m)})) $ as a subspace of $V\otimes \C(\xi)$.

For any complex number $c$, one can easily prove by induction on $i$ the following relation: 
\begin{equation}\label{rel 4.1}
(u-c)^{\overline{i}}d^{i}=\prod_{j=1}^{i}((u-c)d-j+1).
\end{equation}

It follows from Theorem \ref{duality rep} and formula \eqref{rel 4.1} that we can find elements $C_{ij}\in \overline{\mathcal{B}}_{n}(b,w;a,\lambda^{(n)})$, $i=0\lc m$, $j=0\lc n$ such that 
\begin{equation}\label{C ij}
u^{\underline{n}}\prod_{i=1}^{m}(u-w_{i})\,\,D_{n}(b,w;a,\lambda^{(n)})=\sum_{i=0}^{m}\sum_{j=0}^{n}C_{ij}(u+m)^{i}(u-n+1)^{\overline{j}}d^{j}.
\end{equation}

\begin{prop}\label{tG}
For each $i=0\lc m$, $j=0\lc n$, we have $\deg (C_{ij})\leq m-i$, and the following holds:
\begin{equation}\label{Prop 4.2}
\sum_{i=0}^{m}\sum_{j=0}^{n} \, gr_{m-i}\, (C_{ij})\,u^{i+j}\partial_{u}^{j}=u^{n}\prod_{i=1}^{m}(u-\xi_{i})\,\,\mathfrak{D}_{n}^{tG}(b,\xi;a,\lambda^{(n)}).
\end{equation}
\end{prop}
\begin{proof}
    If $\mathcal{A}$ is a $\Z$-filtered algebra with filtration components $\mathcal{F}_{i}$, $i\in\Z$, let us define a filtration on the algebra $D(\mathcal{A})$ of difference operators with coefficients in $\mathcal{A}((u^{-1}))$ as follows. Consider elements of $D(\mathcal{A})$ as series in $u^{-1}$ and $\partial_{u}$ by identifying $\tau$ and $e^{\partial_{u}}=\sum_{k=0}^{\infty}\frac{\partial_{u}^{k}}{k!}$. Then the filtration components $\mathcal{F}^{D}_{i}$, $i\in\Z$ of $D(\mathcal{A})$ are given by
    \[\mathcal{F}^{D}_{i}=\{\sum_{k,l}a_{kl}u^{-k}\partial_{u}^{l}\,\vert\,a_{kl}\in\mathcal{F}_{i+k+l}\text{ for all }k,l\},\]
    in particular, $\deg\, u=1$ and $\deg\, d = - 1$ Also, we have $gr (D(\mathcal{A}))\subset\mathfrak{D}(gr\, A)$, where $\mathfrak{D}(gr\, A)$ is the algebra of differential operators with coefficients in $(gr\, A)((u^{-1}))$, see Section \ref{trigonometric Gaudin} for notation.

    Introduce a filtration on $\Yn$ by defining the filtration components $\mathcal{F}_{i}^{Y}\subset \Yn$, $i\in\Z$, as follows
    \[\mathcal{F}^{Y}_{i}=span_{\C}\{t_{i_{1}j_{1}}^{(r_{1})}t_{i_{2}j_{2}}^{(r_{2})}\dots t_{i_{k}j_{k}}^{(r_{k})}\,\vert\, \sum_{s=1}^{k}r_{s}\leq i+k\}.\]
    Then $gr\,\Yn = U(\gln [t])$ (see, for example, Section 1.5 in \cite{Mol2}). Recall the difference operator $D_{n}\in D(\Yn)$ introduced in Section \ref{Bethe subalgebras in Yangian}. Theorem 10.2 in \cite{MTV6} states that $\deg D_{n}=-n$ and
    \begin{equation}\label{deg D}
        gr\, D_{n}=\mathfrak{D}_{n}, 
    \end{equation}
    where $\mathfrak{D}_{n}$ is the differential operator from Section \ref{trigonometric Gaudin}.

    Notice that $ev_{n,c}(t_{i_{1}j_{1}}^{(r_{1})}t_{i_{2}j_{2}}^{(r_{2})}\dots t_{i_{k}j_{k}}^{(r_{k})})$ is a polynomial in $c$ of degree $\sum_{s=1}^{k}r_{s}-k$. Therefore, the filtration on $\Yn$ is consistent with filtration \eqref{filtr lambda} on $\End(S^{(n)}(b)[a])\otimes\C(\lambda^{(m)})$, and formula \eqref{deg D} implies 
    \begin{equation}\label{deg RHS}
        gr\,\left(u^{\underline{n}}\prod_{i=1}^{m}(u-w_{i})\,\,D_{n}(b,w;a,\lambda^{(n)})\right)=u^{n}\prod_{i=1}^{m}(u-\xi_{i})\,\,\mathfrak{D}_{n}^{tG}(b,\xi;a,\lambda^{(n)}).
    \end{equation}

    On the other hand, using that $\deg D_{n}=-n$ and $\deg\, u=1$, we see that the degree of the left hand side of \eqref{C ij} is equal to $m$. Therefore, the degree of each term in the right hand side of \eqref{C ij} is less or equal then $m$, which gives $\deg (C_{ij})\leq m-i$, and we have
    \begin{equation}\label{deg LHS}
    \begin{split}
        gr\,\left(u^{\underline{n}}\prod_{i=1}^{m}(u-w_{i})\,\,D_{n}(b,w;a,\lambda^{(n)})\right) & =\sum_{j=0}^{n}\sum_{i=0}^{n}\,gr_{m}\left(C_{ij}(u+m)^{i}(u-n+1)^{\overline{j}}d^{j}\right)=\\
        & = \sum_{i=0}^{m}\sum_{j=0}^{n} \, gr_{m-i}\, (C_{ij})\,u^{i+j}\partial_{u}^{j}.
        \end{split}
    \end{equation}

    Comparing formulas \eqref{deg RHS} and \eqref{deg LHS}, we see that the proposition is proved.
\end{proof}

\subsection{Degeneration to XXX spin chain model}
Let us continue working with $\lambda^{(m)}$ as the set of commuting variables. We would like to consider the algebra $\overline{\mathcal{B}}_{m}(a,z;b,\lambda^{(m)})$ as a subalgebra of $\End(S^{(m)}(a)[b])\otimes\C(\lambda^{(m)})$. For that, we have to replace in the construction of $\overline{\mathcal{B}}_{m}(a,z;b,\lambda^{(m)})$ the $\glm$-module $M_{\lambda^{(m)}}^{(m)}$ with the module $M_{(\lambda^{(m)})}^{(m)}=U(\mathfrak{n}_{-})\otimes\C(\lambda^{(m)})$, where $\mathfrak{n}_{-}$ is the Lie subalgebra of lower-triangular matrices in $\glm$, and the action of $\glm$ on $M_{(\lambda^{(m)})}^{(m)}$ is defined as follows. First, consider the $\glm$-module $\Um\otimes\C(\lambda^{(m)})$, where $\glm$ acts on the first tensor factor by multiplication from the left and on the second tensor factor trivially. Then identify $M_{(\lambda^{(m)})}^{(m)}$ with the quotient $(\Um\otimes\C(\lambda^{(m)}))/J$, where $J$ is the left ideal generated by $\mathfrak{n}_{-}$ and $e^{(m)}_{ii}-\lambda_{n+i}$, $i=1\lc m$. With this change, $\overline{\mathcal{B}}_{m}(a,z;b,\lambda^{(m)})$ becomes a subalgebra of $\End(S^{(m)}(a)[b])\otimes\C(\lambda^{(m)})$. On the latter algebra, we consider the filtration described in the previous section, see \eqref{filtr lambda}.

It follows from Theorem \ref{duality rep} that we can find elements $D_{ij}\in \overline{\mathcal{B}}_{m}(a,z;b,\lambda^{(m)})$, $i=0\lc m$, $j=0\lc n$ such that 
\begin{equation}\label{D ij}
u^{\underline{m}}\prod_{i=1}^{n}(u-z_{i})\,\,D_{m}(a,z;b,\lambda^{(m)})=\sum_{i=0}^{m}\sum_{j=0}^{n}D_{ij}(u+n)^{\underline{j}}((u-m+1)d)^{i}.
\end{equation}

\begin{prop}\label{XXX}
For each $i=0\lc m$, $j=0\lc n$, we have $\deg (D_{ij})\leq m-i$, and the following holds:
\begin{equation}\label{gr D ij}
\sum_{i=0}^{m}\sum_{j=0}^{n} \, gr_{m-i}\, (D_{ij})\,(u+n)^{\underline{j}}\tau^{i}=\prod_{i=1}^{n}(u-z_{i})\,\,D_{m}^{XXX}(a,\lambda^{(n)};b,\xi).
\end{equation}
\end{prop}
\begin{proof}
In this proof, for any $\Z$-filtered algebra $\mathcal{A}$ with filtration components $\mathcal{F}_{i}$, $i\in\Z$, we will use the following filtration on $\mathcal{A}((u^{-1}))$: the corresponding filtration components $\widetilde{\mathcal{F}}_{i}$, $i\in\Z$ are given by
\[\widetilde{\mathcal{F}}_{i}=\{\sum_{l}f_{l}u^{l}\,\vert\,f_{l}\in\mathcal{F}_{i}\}.\]
In particular, $\deg u = 0$, note the difference with the proof of the previous proposition.

    Fix $J=\{j_{1}<j_{2}<\dots <j_{k}\}\subset\{1\lc m\}$. Apply the comultiplication $\Delta$ to a quantum minor $t_{J}(u)\in\Ym ((u^{-1}))$. We get 
    \begin{equation}\label{delta t J}
    \begin{split}
        \Delta (t_{J}(u))=
        \sum_{\sigma\in S_{k}}(-1)^{\sigma}\sum_{i_{1},i_{2}\lc i_{k}=1}^{m}\bigl(
        t_{j_{\sigma(1)}i_{1}}(u)t_{j_{\sigma(2)}i_{2}}(u-1)\dots t_{j_{\sigma(k)}i_{k}}(u-k+1)\otimes\\
        \otimes t_{i_{1}j_{1}}(u)t_{i_{2}j_{2}}(u-1)\dots t_{i_{k}j_{k}}(u-k+1)\bigr).
        \end{split}
    \end{equation}

    Let us focus on the second factor in the tensor product in formula \eqref{delta t J}. Let $\rho_{(\lambda^{(m)})}:\Um\rightarrow \End (U(\mathfrak{n}_{-}))\otimes\C (\lambda^{(m)})$ be the homomorphism corresponding to the action of $\glm$ on $M_{(\lambda^{(m)})}^{(m)}$. Using the filtration on $\End (U(\mathfrak{n}_{-}))\otimes\C (\lambda^{(m)})$ defined in formula \eqref{filtr lambda}, we argue that the degree of
    \[\rho_{(\lambda^{(m)})}\circ\, ev_{m}\bigl(t_{i_{1}j_{1}}(u)t_{i_{2}j_{2}}(u-1)\dots t_{i_{k}j_{k}}(u-k+1)\bigr)\]
    is maximal and equal to $k$ if and only if $i_{s}=j_{s}$ for all $s=1\lc k$ and
    \begin{equation}\label{gr t}
        gr_{k}\bigl(\rho_{(\lambda^{(m)})}\circ\, ev_{m}\bigl(t_{j_{1}j_{1}}(u)t_{j_{2}j_{2}}(u-1)\dots t_{j_{k}j_{k}}(u-k+1)\bigr)\bigr)=\frac{\prod_{s=1}^{k}\xi_{j_{s}}}{u^{\underline{k}}}\cdot\id,
    \end{equation}
    where $\id$ is the identity transformation of $U(\mathfrak{n}_{-})$.
    Indeed, this follows from two observations. Firstly, for any $i=1\lc m$, we have
    \[\rho_{(\lambda^{(m)})}\circ\, ev_{m}\bigl(t_{ii}(u)\bigr)\cdot (1\otimes 1)= 1\otimes \left(1+\frac{\xi_{i}}{u}\right).\]
    Secondly, for any $x\in\mathfrak{n}_{-}$, we have 
    \begin{equation*}
    \begin{split}
    & \rho_{(\lambda^{(m)})}\circ\, ev_{m}\bigl(t_{j_{1}j_{1}}(u)t_{j_{2}j_{2}}(u-1)\dots t_{j_{k}j_{k}}(u-k+1)\bigr)\cdot (x\otimes 1)= \\
    & = x\bigl(\rho_{(\lambda^{(m)})}\circ\, ev_{m}\bigl(t_{j_{1}j_{1}}(u)t_{j_{2}j_{2}}(u-1)\dots t_{j_{k}j_{k}}(u-k+1)\bigr)\cdot (1\otimes 1)\bigr) +\text{terms of smaller degree}.
    \end{split}
    \end{equation*}

    For any $k=0\lc m$, let $\bar{B}_{k}(u)$ and $B^{XXX}_{k}(u)$ be the coefficients of the difference operators $D_{m}(a,z;b,\lambda^{(m)})$ and $D_{m}^{XXX}(a,\lambda^{(n)};b,\xi)$, respectively:
    \[D_{m}(a,z;b,\lambda^{(m)})=\sum_{k=0}^{m}\bar{B}_{k}(u)\tau^{m-k},\quad D_{m}^{XXX}(a,\lambda^{(n)};b,\xi)=\sum_{k=0}^{m}B^{XXX}_{k}(u)\tau^{m-k}.\]
    Then, from the construction of the difference operators $D_{m}(a,z;b,\lambda^{(m)})$ and $D_{m}^{XXX}(a,\lambda^{(n)};b,\xi)$ and formula \eqref{B=t}, it follows that
    \[\bar{B}_{k}(u)=\widehat{\theta^{S^{(m)}(a)}_{\lambda^{(m)},b}}\left(\sum_{J, |J|=k}\rho_{a,z}\otimes(\rho_{(\lambda^{(m)})}\circ\, ev_{m})(\Delta (t_{J}(u)))\right),\]
    \[B^{XXX}_{k}(u)=\sum_{J, |J|=k}\left(\prod_{j\in J}\xi_{j}\right)\rho_{a,z}(t_{J}(u)),\]
    where $\rho_{a.z}$ is the homomorphism corresponding to $\Ym$-module $S^{(m)}(a,z)$, see formula \eqref{S(a,z)}.
    
    Denote $X_{k}(u)=\sum_{J, |J|=k}\rho_{a,z}\otimes(\rho_{(\lambda^{(m)})}\circ\, ev_{m})(\Delta (t_{J}(u)))$. Then formulas \eqref{delta t J} and \eqref{gr t} imply
    \[gr_{k}(X_{k}(u))=\frac{1}{u^{\underline{k}}}B_{k}^{XXX}(u)\otimes \id.\]

    In this proof, to simplify the notation, let us denote $\theta\coloneqq\theta^{S^{(m)}(a)}_{\lambda^{(m)},b}$. Notice that we can extend $\theta$ to the map $\bigl(\bigotimes_{i=1}^{n}S^{a_{i}}\C^{m}\otimes M^{(m)}_{(\lambda^{(m)})}\bigr)[\lambda^{(m)}+b]\rightarrow S^{(m)}(a)[b]\otimes\C(\lambda^{(m)})$ using the same formula \eqref{theta} with $\theta^{-1}$ being its \textit{right} inverse. We will use it on the step 4 in the computation \eqref{gr Bv} below.

    It is also proved in \cite[Theorem 4]{Ryb2} that for any $v\in S^{(m)}(a)[b]\otimes\C(\lambda^{(m)})$, we have
    \[\theta^{-1}v=v\otimes v_{\lambda^{(m)}}+\text{terms of smaller degree},\]
    where $v_{\lambda^{(m)}}$ is a highest weight vector in $M^{(m)}_{(\lambda^{(m)})}$. We will use it on the step 5 in the computation \eqref{gr Bv} below.

    Fix a vector $v\in S^{(m)}(a)[b]\otimes\C(\lambda^{(m)})$ of degree $0$. In the next computation, we work with series in $u^{-1}$ assuming that all operations are applied simultaneously to all coefficients of a series.
    \begin{equation}\label{gr Bv}
        \begin{split}
            & gr_{k}(\bar{B}_{k}(u))v = gr_{k}(\bar{B}_{k}(u)v) = gr_{k}((\widehat{\theta}X_{k}(u))v) =\\
            & = gr_{k}(\theta X_{k}(u)\theta^{-1}v) = \theta\, gr_{k}(X_{k}(u)\theta^{-1}v) = \theta\, \bigl( gr_{k}(X_{k}(u))\bigr)(v\otimes v_{\lambda^{(m)}}) = \\
            & = \theta\left(\frac{1}{u^{\underline{k}}}B_{k}^{XXX}(u)\otimes \id\right)(v\otimes v_{\lambda^{(m)}})=\theta\left(\frac{1}{u^{\underline{k}}}B_{k}^{XXX}(u)v\otimes v_{\lambda^{(m)}}\right) = \frac{1}{u^{\underline{k}}}B_{k}^{XXX}(u)v.
        \end{split}
    \end{equation}
    
    Therefore, it follows that
    \begin{equation}\label{gr B}
        gr_{k}\bigl(\bar{B}_{k}(u)\bigr)=\frac{1}{u^{\underline{k}}}B_{k}^{XXX}(u).
    \end{equation}

    Now, let us consider the expansion of $D_{m}(a,z;b,\lambda^{(m)})$ in powers of $d$ rather then in powers of $\tau$, and let $\widetilde{B}_{k}$, $k=0\lc m$ be the corresponding coefficients: 
    \[D_{m}(a,z;b,\lambda^{(m)})=\sum_{k=0}^{m}\widetilde{B}_{k}(u)d^{m-k}.\]
    Then, for each $k=0\lc m$, we have $\widetilde{B}_{k}(u)=\bar{B}_{k}(u)+\sum_{k'=0}^{k-1}c_{k'k}\bar{B}_{k'}(u)$ for some complex numbers $c_{k'k}$. in particular, $\deg \widetilde{B}_{k}(u)=\deg \bar{B}_{k}(u)=k$ and $gr_{k} (\widetilde{B}_{k}(u))=gr_{k} (\bar{B}_{k}(u))$.

    Also, using formula \eqref{rel 4.1}, we can find elements $\widetilde{D}_{ij}\in\overline{\mathcal{B}}_{m}(a,z;b,\lambda^{(m)})$, $i=0\lc m$, $j=0\lc n$ such that 
\begin{equation*}
u^{\underline{m}}\prod_{i=1}^{n}(u-z_{i})\,\,D_{m}(a,z;b,\lambda^{(m)})=\sum_{i=0}^{m}\sum_{j=0}^{n}\widetilde{D}_{ij}(u+n)^{\underline{j}}(u-m+1)^{\overline{i}}d^{i}.
\end{equation*}
Then, for any $k=0\lc m$, we have 
\begin{equation}\label{BD}
u^{\underline{m}}\prod_{i=1}^{n}(u-z_{i})\widetilde{B}_{m-k}(u)=\sum_{j=0}^{n}\widetilde{D}_{kj}(u+n)^{\underline{j}}(u-m+1)^{\overline{k}},
\end{equation}
which implies $\deg \widetilde{D}_{kj}\leq m-k$ for any $j=0\lc n$. It is also easy to see that $D_{kj}=\widetilde{D}_{kj}+\sum_{k'=k+1}^{m}c'_{k'k}\widetilde{D}_{k'j}$ for some complex numbers $c'_{k'k}$. Therefore, $\deg D_{kj}\leq m-k$ and $gr_{m-k}(\widetilde{D}_{kj})=gr_{m-k}(D_{kj})$. Thus, identity \eqref{BD} gives
\begin{equation}\label{BD2}
    u^{\underline{m}}\prod_{i=1}^{n}(u-z_{i})gr_{m-k} (\bar{B}_{m-k}(u))=\sum_{j=0}^{n}gr_{m-k}(D_{kj})(u+n)^{\underline{j}}(u-m+1)^{\overline{k}}
\end{equation}
Finally, using \eqref{gr B} and \eqref{BD2}, we compute
\begin{equation*}
    \begin{split}
        \prod_{i=1}^{n}(u-z_{i})\,\,D_{m}^{XXX}(a,\lambda^{(n)};b,\xi)=u^{\underline{m}}\prod_{i=1}^{n}(u-z_{i})\sum_{k=0}^{m}\frac{B_{k}^{XXX}(u)}{u^{\underline{k}}}\cdot\frac{1}{(u-m+1)^{\overline{m-k}}}\tau^{m-k}=\\
        =\sum_{k=0}^{m}\left(u^{\underline{m}}\prod_{i=1}^{n}(u-z_{i})gr_{k}(\bar{B}_{k}(u))\right)\cdot\frac{1}{(u-m+1)^{\overline{m-k}}}\tau^{m-k}=\\
        =\sum_{k=0}^{m}\left(\sum_{j=0}^{n}gr_{m-k}(D_{kj})(u+n)^{\underline{j}}\right)\tau^{k},
    \end{split}
\end{equation*}
which finishes the proof of the proposition.
\end{proof}

\subsection{Degeneration of the duality map $\Psi_{\lambda}$}
\subsubsection{} Let $I$ be a totally ordered set, and let $\mathbb{F}$ be a field extension of $\C$. Suppose that as a complex vector space, $\mathbb{F}$ is $I$-filtered. For any other complex vector space $V$, introduce an $I$-filtration $\{\mathcal{F}_{i}^{V\otimes\mathbb{F}}, i\in I\}$ on $V\otimes \mathbb{F}$ generalizing the filtration \eqref{filtr lambda}, that is 
\begin{equation}\label{filtr f}
    \mathcal{F}_{i}^{V\otimes\mathbb{F}}=span_{\C}\{v\otimes f\,\vert\, \deg (f)\leq i\}.
\end{equation}
Let $V$ and $W$ be two complex vector spaces, and let $\phi$ be an $\mathbb{F}$-linear space from $V\otimes\mathbb{F}$ to $W\otimes\mathbb{F}$. Considering $\phi$ as an element of $\Hom_{\C}(V,W)\otimes \mathbb{F}$ and using the filtration on the latter space defined by \eqref{filtr f}, we denote $\phi^{gr}=gr (\phi)$.

We say that $\phi$ preserves the degree if $\deg (\phi(p))=\deg (p)$ for any $p\in V\otimes\mathbb{F}$. We will need the following three lemmas.
\begin{lem}\label{lemma1}
Suppose that $\phi$ preserves the degree. Then for any $k\in I$ and $p\in V\otimes\mathbb{F}$ such that $\deg (p)\leq k$, we have $gr_{k}(\phi(p))=\phi^{gr}(gr_{k} (p))$.
\end{lem}
\begin{proof}
    It is clear that for any $v\in V$ and $k\geq \deg\phi$, we have $\deg (\phi (v))=\deg \phi$ and $gr_{k}(\phi(v))=(gr_{k}\phi)(v)$. (Here and below, we identify $V$ and $V\otimes 1$).
    In particular, if $\phi$ preserves the degree, then $\deg (\phi) = 0$, and for any $v\in V$, we have $gr_{0}(\phi(v))=\phi^{gr}(v)$.  
    
    Let $k\in I$ and $p\in V\otimes\mathbb{F}$ be such that $\deg (p)\leq k$.
    We can write $p=\sum_{i=1}^{l}v_{i}\otimes f_{i}$ for some $v_{1}\lc v_{l}\in V$, $f_{1}\lc f_{l}\in\mathbb{F}$ such that $\deg (f_{i})\leq k$ for all $i=1\lc l$. Then we compute
    \[gr_{k}(\phi(p))= gr_{k}\left(\sum_{i=1}^{l}\phi(v_{i}) f_{i}\right)= \sum_{i=1}^{l}gr_{0}(\phi(v_{i})) gr_{k}(f_{i})=\sum_{i=1}^{l}\phi^{gr}(v_{i})gr_{k}(f_{i})=\phi^{gr}(gr_{k}(p)).\]
\end{proof}

\begin{lem}\label{lemma2}
    Suppose that $\phi$ preserves the degree and that $\phi$ is an isomorphism.
    Then $\phi^{gr}$ is also an isomorphism and $(\phi^{gr})^{-1}=(\phi^{-1})^{gr}$.
\end{lem}
\begin{proof}
    First, observe that $\phi^{-1}$ preserves the degree: for any $p\in W\otimes\mathbb{F}$, we have $\deg (\phi^{-1}p)=\deg (\phi (\phi^{-1}(p))=\deg(p)$. Then, using the previous lemma, for any $v\in V$, we get
    \[v=gr_{0}v=gr_{0}(\phi^{-1}(\phi(v))=(\phi^{-1})^{gr}\phi^{gr}gr_{0}(v)=(\phi^{-1})^{gr}\phi^{gr}v.\]
    Since both $id_{V\otimes\mathbb{F}}$ and $(\phi^{-1})^{gr}\phi^{gr}$ are $\mathbb{F}$-linear, the previous computation implies $(\phi^{-1})^{gr}\phi^{gr}=id_{V\otimes\mathbb{F}}$. Similarly, we get $\phi^{gr}(\phi^{-1})^{gr}=id_{W\otimes\mathbb{F}}$.
\end{proof}
\begin{lem}\label{lemma3}
Suppose that $\phi$ preserves the degree and that $\phi$ is an isomorphism. Let $X\in\End(V)\otimes\mathbb{F}$, $Y\in\End (W)\otimes\mathbb{F}$, and $k\in I$ be such that $k\geq \deg (X)$, $k\geq \deg (Y)$. Suppose that $\widehat{\phi}(X)=Y$. Then $\widehat{\phi^{gr}}(gr_{k}\,X)= gr_{k}\,Y$. 
\end{lem}
\begin{proof}
    For any $w\in W$, using Lemmas \ref{lemma1} and \ref{lemma2}, we compute
    \begin{equation*}
        \begin{split}
            (gr_{k}\, Y)w = gr_{k}\, (Yw) = gr_{k}\, (\phi X \phi^{-1} w) = \phi^{gr}\,gr_{k}\,(X \phi^{-1} w) = \\
             = \phi^{gr} (gr_{k}\, X)(gr_{0}\, (\phi^{-1}w)) = \phi^{gr} (gr_{k}\, X) (\phi^{gr})^{-1}w = \widehat{\phi^{gr}}(gr_{k}\,X) w. 
        \end{split}
    \end{equation*}
    Since both $gr_{k}\, Y$ and $\widehat{\phi^{gr}}(gr_{k}\,X)$ are $\mathbb{F}$-linear, the previous computation implies $\widehat{\phi^{gr}}(gr_{k}\,X)= gr_{k}\,Y$ and the lemma is proved. 
\end{proof}
\subsubsection{}\label{4.7.2} Now, let us make the necessary adjustments to think of the map $\Psi_{\lambda}$ as a $\C(\lambda^{(m)})$-linear map from $S^{(m)}(a)[b]\otimes \C(\lambda^{(m)})$ to $S^{(n)}(b)[a]\otimes \C(\lambda^{(m)})$. Namely, in the chain \eqref{chain}, we replace $S^{(m)}(a)[b]$, $S^{(n)}(b)[a]$, and $M^{(m)}_{\lambda^{(m)}}$ with $S^{(m)}(a)[b]\otimes \C(\lambda^{(m)})$, $S^{(n)}(b)[a]\otimes \C(\lambda^{(m)})$, and $M^{(m)}_{(\lambda^{(m)})}$, respectively. We will also need to replace the module $M_{\lambda}^{(m+n)}$ with the $\gl_{m+n}$-analog $M_{(\lambda)}^{(m+n)}$ of the module $M^{(m)}_{(\lambda^{(m)})}$ with the difference that in $M_{(\lambda)}^{(m+n)}$, only the last $m$ components of the highest weight $\lambda$ become variables, and the rest components remain complex numbers. With such a change, all maps in \eqref{chain} have the obvious $\C(\lambda^{(m)})$-linear analogs and $\Psi_{\lambda}$ becomes an element of $\Hom (S^{(m)}(a)[b], S^{(n)}(b)[a])\otimes \C(\lambda^{(m)})$. 

Unfortunately, we cannot apply the results of the previous subsection to the map $\Psi_{\lambda}$, because it is not clear whether $\Psi_{\lambda}$ preserves the degree or not. Therefore, we would like to modify the map $\Psi_{\lambda}$ in such a way that the resulting map preserves the degree and satisfies relation \eqref{duality explicit}. For that, we need to diagonalize the algebra $\overline{\mathcal{B}}_{n}(b,w;a,\lambda^{(n)})$, which requires changing the field $\C(\lambda^{(n)})$ to another, algebraically closed field that we introduce below.

Recall the Zarisky open subset $J_{a,b}\in\C^{n+m}$ of all values of $\lambda$ at which all the maps in the chain \eqref{chain} are well defined isomorphisms. Fix some $\lambda^{(n)}$ such that $J_{a,b}^{(m)}\coloneqq\{\lambda_{0}^{(m)}\in\C^{m}\,\vert\, (\lambda^{(n)},\lambda_{0}^{(m)})\in J_{a,b}\}$ is not empty. Let $\C^{reg}(\lambda^{(m)})\subset\C(\lambda^{(m)})$ be the subalgebra of all rational functions which are regular on $J^{(m)}_{a,b}$. Notice that $\Psi_{\lambda}\in \Hom (S^{(m)}(a)[b], S^{(n)}(b)[a])\otimes \C^{reg}(\lambda^{(m)})$, $C_{ij}\in \End (S^{(n)}(b)[a])\otimes \C^{reg}(\lambda^{(m)})$, and  $D_{ij}\in \End (S^{(m)}(a)[b])\otimes \C^{reg}(\lambda^{(m)})$. Recall the homomorphism $\pi_{r,\xi}:\C(\lambda^{(m)})\rightarrow\C(r,\xi)\subset \C(\xi)((r^{-1}))$, $\pi_{r,\xi}(\lambda_{n+i})=r\xi_{i}$. Then $\pi_{r,\xi}(\C^{reg}(\lambda^{(m)}))\subset \C^{reg}(\xi)((r^{-1}))$, where $\C^{reg}(\xi)$ is the algebra of rational functions in $\xi$, which are regular on some Zariski open subset $\widetilde{J}^{(m)}_{a,b}\subset\C^{m}$. Actually, if $J_{a,b}^{(m)}$ is given by inequalities $p(\lambda^{m})\neq 0$ for $p\in P\subset \C[\lambda^{(m)}]$, then $\widetilde{J}_{a,b}^{(m)}$ is given by inequalities $gr\, (p) (\xi)\neq 0$ for $p\in P$.

Let us now replace the elements $\Psi_{\lambda}$, $D_{ij}$, and $C_{ij}$ with the evaluations of their images under the map $\id\otimes\pi_{r,\xi}$ at some value of $\xi$ in $\widetilde{J}_{a,b}^{(m)}$, keeping the same notation. That is, we think of $\xi=(\xi_{1}\lc\xi_{m})$ as the set of complex parameters again, and we consider $\Psi_{\lambda}$, $D_{ij}$, and $C_{ij}$ as elements of $\Hom (S^{(m)}(a)[b], S^{(n)}(b)[a])\otimes \C((r^{-1}))$, $\End (S^{(n)}(b)[a])\otimes \C((r^{-1}))$, and  $\End (S^{(m)}(a)[b])\otimes \C((r^{-1}))$, respectively. So, these elements now belong to vector spaces over the field $\C((r^{-1}))$, the algebraic closure of which is known to be the algebra of Puiseux series, which we denote by $\mathbb{P}_{r}$. An element of $\mathbb{P}_{r}$ is a series of the form
\[\sum_{i\in\Z,i\leq k}c_{i}r^{\frac{i}{l}}\]
for some $k,l\in \Z$, $c_{i}\in\C$, and the addition and multiplication between series are defined in a natural way.

Let us extend the field and think of $\Psi_{\lambda}$, $D_{ij}$, and $C_{ij}$ as elements of $\Hom (S^{(m)}(a)[b], S^{(n)}(b)[a])
\otimes \mathbb{P}_{r}$, $\End (S^{(n)}(b)[a])\otimes \mathbb{P}_{r}$, and  $\End (S^{(m)}(a)[b])\otimes \mathbb{P}_{r}$, respectively. In particular, the algebras $\overline{\mathcal{B}}_{n}(b,w;a,\lambda^{(n)})$, where $w$ is now a linear function in $r$ and $\xi$, and $\overline{\mathcal{B}}_{m}(a,z;b,r\xi)$ are now $\mathbb{P}_{r}$-subalgebras of $\End (S^{(n)}(b)[a])\otimes \mathbb{P}_{r}$ and  $\End (S^{(m)}(a)[b])\otimes \mathbb{P}_{r}$ generated by $C_{ij}$ and $D_{ij}$, respectively. On $\mathbb{P}_{r}$, we have a natural $\mathbb{Q}$-filtration $\{\mathcal{F}_{q}^{\mathbb{P}},q\in\mathbb{Q}\}$, where for each $q\in\mathbb{Q}$,  $\mathcal{F}_{q}^{\mathbb{P}}$ consists of all elements of the form $\sum_{i\in\Z,i\leq ql}c_{i}r^{\frac{i}{l}}$ for some $l\in \Z$, $c_{i}\in\C$. Using \eqref{filtr f}, we extend this filtration to $V\otimes \mathbb{P}_{r}$ for any complex vector space $V$. Notice that Propositions \ref{tG} and \ref{XXX} still hold in the new setting if we make similar adjustments for the constructions of $\mathfrak{D}_{n}^{tG}(b,\xi;a,\lambda^{(n)})$ and $D_{m}^{XXX}(a,\lambda^{(n)};b,\xi)$ as follows. Namely, in the construction we should first treat $\xi$ as the set of variables, then we need to apply the homomorphism $\xi_{i}\mapsto r\xi_{i}$ to each rational function in $\xi$, and look at the result as an element of $\C(\xi)((r^{-1}))$. After that, we evaluate $\xi$ at some value in $\widetilde{J}_{a}$ and extend the field from $\C((r^{-1}))$ to $\mathbb{P}_{r}$.

We will need the following lemma.

\begin{lem}\label{diagonalization}
    There exists an open (in usual topology) subset $U$ of $\C^{n+m}$, such that for any $(\xi,\lambda^{(n)})\in U$, the algebra $\overline{\mathcal{B}}_{n}(b,w;a,\lambda^{(n)})$ is diagonalizable over $\mathbb{P}_{r}$ with one-dimensional eigenspaces.
\end{lem}
\begin{proof}
Since $\mathbb{P}_{r}$ is algebraically closed, we can decompose $S^{(n)}(b)[a]\otimes\mathbb{P}_{r}$ into direct sum of generalized eigenspaces of $\overline{\mathcal{B}}_{n}(b,w;a,\lambda^{(n)})$ over $\mathbb{P}_{r}$. It is easy to check that if $v$ is a generalized eigenvector of $\overline{\mathcal{B}}_{n}(b,w;a,\lambda^{(n)})$, then $gr (v)$ is a common generalized eigenvector of $gr_{m-i}(C_{ij})$, $i=1\lc m$, $j=1\lc n$. Therefore, by Proposition \ref{tG}, $gr (v)$ is a generalized eigenvector of $\mathcal{B}_{n}^{tG}(b,\xi;a,\lambda^{(n)})$. This implies that the dimension of the largest generalized eigenspace of $\overline{\mathcal{B}}_{n}(b,w;a,\lambda^{(n)})$ is less or equal to the dimension of the largest generalized eigenspace of $\mathcal{B}_{n}^{tG}(b,\xi;a,\lambda^{(n)})$. (It can be less since two different eigenvalues can be "glued together" under degeneration). 

By \cite[Theorem 1.3]{IKR}, there exists an open subset $U$ of $\R^{n+m}$, such that for any $(\xi,\lambda^{(n)})\in U$, the algebra $\mathcal{B}_{n}^{tG}(b,\xi;a,\lambda^{(n)})$ is diagonalizable with one-dimensional eigenspaces, or, equivalently, has one-dimensional generalized eigenspaces. Therefore, it follows that for $(\xi,\lambda^{(n)})\in U$, $\overline{\mathcal{B}}_{n}(b,w;a,\lambda^{(n)})$ has one-dimensional generalized eigenspaces, and the lemma is proved.
\end{proof}

We are now ready to define the modification of the map $\Psi_{\lambda}$, which preserves the degree. Denote $N=\dim S^{(n)}(b)[a]$. Assume that $(\xi,\lambda^{(n)})\in U$. By Lemma \ref{diagonalization}, there exists a $\mathbb{P}_{r}$-basis $\{v_{1}\lc v_{N}\}$ of $S^{(n)}(b)[a]\otimes\mathbb{P}_{r}$ in which any element of $\overline{\mathcal{B}}_{n}(b,w;a,\lambda^{(n)})$ is diagonal. We can assume that $\deg (v_{i})=0$ for any $i=1\lc N$. Notice that by Theorem \ref{duality rep},  $\{\Psi_{\lambda}^{-1}(v_{1})\lc \Psi_{\lambda}^{-1}(v_{N})\}$ is a basis of $S^{(m)}(a)[b]\otimes\mathbb{P}_{r}$, which diagonalizes $\overline{\mathcal{B}}_{m}(a,z;b,\lambda^{(m)})$. For each $i=1\lc N$, denote $d_{i}=\deg \Psi_{\lambda}^{-1}(v_{i})$. Then we define a $\mathbb{P}_{r}$-linear map $\Psi_{(\lambda)}: S^{(m)}(a)[b]\otimes\mathbb{P}_{r}\rightarrow S^{(n)}(b)[a]\otimes\mathbb{P}_{r}$ as follows: $\Psi_{(\lambda)}(r^{-d_{i}}\Psi_{\lambda}^{-1}(v_{i}))=v_{i}$, for any $i=1\lc N$. It follows from the construction that $\Psi_{(\lambda)}$ preserves the degree and that for any $X\in\overline{\mathcal{B}}_{m}(a,z;b,\lambda^{(m)})$, we have $\widehat{\Psi}_{(\lambda)}(X)=\widehat{\Psi}_{\lambda}(X)$. 

For any $i=1\lc m$, $j=1\lc n$, by the definition of the elements $C_{ij}$ and $D_{ij}$, see formulas \eqref{C ij} and \eqref{D ij}, respectively, and also by relations \eqref{rel 4.1} and \eqref{duality explicit}, we have $\widehat{\Psi}_{\lambda}(D_{ij})=C_{ij}$. Therefore, $\widehat{\Psi}_{(\lambda)}(D_{ij})=C_{ij}$, and, applying Lemma \ref{lemma3} to $\widehat{\Psi}_{(\lambda)}$, we get
\begin{equation}\label{deg duality}
\widehat{\Psi^{gr}_{(\lambda)}}(gr_{m-i}\,D_{ij})=gr_{m-i}\,C_{ij}.
\end{equation}

Finally, since we do not need the variable $r$ anymore, let us put $r =1$ in any element of any vector space over $\mathbb{P}_{r}$ that we had above, where it makes sense. In particular, $\Psi^{gr}_{(\lambda)}$, $gr_{m-i}\,C_{ij}$, and $gr_{m-i}\,D_{ij}$ become elements of $\Hom(S^{(m)}(a)[b],S^{(n)}(b)[a])$, $\End (S^{(n)}(b)[a])$, and $\End (S^{(m)}(a)[b])$, respectively, while the algebras  $\mathcal{B}_{m}^{XXX}(a,\lambda^{(n)};b,\xi)$ and $\mathcal{B}_{n}^{tG}(b,\xi;a,\lambda^{(n)})$ become the subalgebras of $\End (S^{(m)}(a)[b])$ and $\End (S^{(n)}(b)[a])$, respectively, as it was originally in the formulation of Theorem \ref{main3}. 
\subsection{Trigonometric Gaudin and dynamical Hamiltonians}
For each $i=1\lc m$ and an element $g\in\gln$, denote $(g)_{(i)}=1\otimes\dots\otimes 1\otimes g\otimes 1\dots\otimes 1\in\Un^{\otimes m}$, where $g$ is on the $i$-th place. Recall that $e_{ij}^{(n)}$, $i,j=1\lc n$, are the standard generators of $\Un$. We will also identify $\Un$ with its image in $\Un^{\otimes m}$ under the homomorphism $g\mapsto \sum_{i=1}^{m}(g)_{(i)}$, $g\in\gln$. 

Assume that $\xi_{i}\neq\xi_{k}$ for $i\neq k$. Introduce the trigonometric Gaudin Hamiltonians $H^{(n)}_{i}(z,\xi)\in \Un^{\otimes m}$, $i=1\lc m$ (see \cite{J}):
\[H^{(n)}_{i}(z,\xi)=\sum_{j=1}^{n}\left(z_{j}+n-\frac{e_{jj}^{(n)}}{2}\right)(e_{jj}^{(n)})_{(i)}+\sum_{\substack{k=1\\k\neq i}}^{m}\frac{\xi_{i}\Omega^{+}_{ik}+\xi_{k}\Omega^{-}_{ik}}{\xi_{i}-\xi_{k}},\]
where
\[\Omega^{+}_{ik}=\frac{1}{2}\sum_{j=1}^{n}(e_{jj}^{(n)})_{(i)}(e_{jj}^{(n)})_{(k)}+\sum_{1\leq j<l\leq n}(e_{jl}^{(n)})_{(i)}(e_{lj}^{(n)})_{(k)}\]
and 
\[\Omega^{-}_{ik}=\frac{1}{2}\sum_{j=1}^{n}(e_{jj}^{(n)})_{(i)}(e_{jj}^{(n)})_{(k)}+\sum_{1\leq j<l\leq n}(e_{lj}^{(n)})_{(i)}(e_{jl}^{(n)})_{(k)}.\]

Using the similar notation, introduce the XXX dynamical Hamiltonians $G^{(m)}_{i}(z,\xi)\in \Um^{\otimes n}$, $i=1\lc m$ (see \cite{TV2}):
\begin{equation*}
\begin{split}    
G^{(m)}_{i}(z,\xi)=-\frac{(e^{(m)}_{ii})^{2}}{2}+\sum_{j=1}^{n}(z_{i}+n)(e_{ii}^{(m)})_{(j)} & +\sum_{j=1}^{m}\sum_{1\leq k<l\leq n}(e_{ij}^{(m)})_{(k)}(e_{ji}^{(m)})_{(l)}+\\
& +\sum_{\substack{j=1\\j\neq i}}^{m}\frac{\xi_{j}}{\xi_{i}-\xi_{j}}(e^{(m)}_{ij}e^{(m)}_{ji}-e^{(m)}_{ii}).
\end{split}
\end{equation*}

Algebras $\Un^{\otimes m}$ and $\Um^{\otimes n}$ act on $S^{(n)}(b)[a]$ and $S^{(m)}(a)[b]$, respectively. We denote the corresponding images of $H^{(n)}_{i}(z,\xi)$ and $G^{(m)}_{i}(z,\xi)$ in $\End (S^{(n)}(b)[a])$ and $\End (S^{(m)}(a)[b])$ by $\overline{H}^{(n)}_{i}(z,\xi)$ and $\overline{G}^{(m)}_{i}(z,\xi)$, respectively. 

Recall the map $\Psi$, see formula \eqref{Psi}. It was observed more then twenty years ago (see \cite{TV4}) that for each $i=1\lc m$, we have:
\begin{equation}\label{ancient duality}
    \widehat{\Psi}(\overline{G}^{(m)}_{i}(z,\xi))=\overline{H}^{(n)}_{i}(z,\xi).
\end{equation}

Recall the elements $gr_{m-i}C_{ij}\in\End (S^{(n)}(b)[a])$ and $gr_{m-i}D_{ij}\in\End (S^{(m)}(a)[b])$ from the previous section. For each $i=1\lc m$, denote 
\[C_{i}(u)=\frac{\sum_{j=1}^{n}(gr_{m-i}C_{ij})u^{j}}{\prod_{k=1}^{m}(u-\xi_{k})}, \quad D_{i}(u)=\frac{\sum_{j=1}^{n}(gr_{m-i}D_{ij})u^{j}}{\prod_{k=1}^{m}(u-\xi_{k})}.\]

\begin{thm}\label{Hamiltonians through coefficients}
    For each $i=1\lc m$, the following relations hold:
    \begin{equation}\label{trig Gaudin Hamilt}
    \overline{H}^{(n)}_{i}(z,\xi)=\frac{1}{\xi_{i}}\Res_{u=\xi_{i}}\left(\frac{1}{2}C^{2}_{1}(u)-C_{2}(u)\right)-\frac{1}{2}b^{2}_{i},
    \end{equation}
    \begin{equation}\label{dynamical Hamilt}
    \overline{G}^{(m)}_{i}(z,\xi)=\frac{1}{\xi_{i}}\Res_{u=\xi_{i}}\left(\frac{1}{2}D^{2}_{1}(u)-D_{2}(u)\right)-\frac{1}{2}b^{2}_{i}.
    \end{equation}
\end{thm}
\begin{proof}
    Given that by Proposition \ref{XXX}, the elements $gr_{m-i}(D_{ij})$, $i=1\lc m$, $j=1\lc n$, are the coefficients of the difference operator $\prod_{i=1}^{n}(u-z_{i})\,\,D_{m}^{XXX}(a,\lambda^{(n)};b,\xi)$, relation \eqref{dynamical Hamilt} can be easily deduced from Proposition B.1 in \cite{MTV6}.

    Let us prove relation \eqref{trig Gaudin Hamilt}. For each $i,j=1\lc n$, define differential operators $\bar{\mathfrak{D}}_{ij}\in\mathfrak{D}(\Un^{\otimes m+1})$ as follows:
    \[\overline{\mathfrak{D}}_{ij}=\delta_{ij}(u\partial_{u}-n+j)-\sum_{l=1}^{m}\frac{u(e_{ij}^{(n)})_{(l)}}{u-\xi_{l}}-(e_{ij}^{(n)})_{(m+1)}.\]
    Also, denote 
    \[\overline{\mathfrak{D}}_{n}= \sum_{\sigma\in S_{n}}(-1)^{\sigma}\overline{\mathfrak{D}}_{\sigma (1),1}\overline{\mathfrak{D}}_{\sigma (2), 2}\dots \overline{\mathfrak{D}}_{\sigma (n), n}.\]

    Recall the differential operator $\mathfrak{D}_{n}$ and the $U(\gln [t])$-module $S_{U}^{(n)}(b,\xi,\lambda^{(n)})$ (see formulas \eqref{mathfrak D} and \eqref{S_U}, respectively). It is easy to check that for any $v\in S_{U}^{(n)}(b,\xi,\lambda^{(n)})$, we have $u^{n}\mathfrak{D}_{n}\cdot v=\overline{\mathfrak{D}}_{n}\cdot v$.
    
    Assume that $v\in \bigl(S^{(n)}(b)\otimes M_{\lambda^{(n)}}^{(n)}\bigr)^{sing}[\lambda^{(n)}+a]$. Then $\theta^{S^{(n)}(b)}_{\lambda^{(n)},a}\bigl(\overline{\mathfrak{D}}_{n}\cdot v\bigl) = \widetilde{\mathfrak{D}}_{n}\cdot \theta^{S^{(n)}(b)}_{\lambda^{(n)},a}(v)$, where the differential operator $\widetilde{\mathfrak{D}}_{n}\in \mathfrak{D}$$(\Un^{\otimes m})$ is obtained from $\overline{\mathfrak{D}}_{n}$ as follows. Let $\overline{\Un^{\otimes m+1}}$ be the quotient of $\Un^{\otimes m+1}$ by the left ideal generated by $(e^{(n)}_{ij})_{(m+1)}$ for all $i>j$, $(e^{(n)}_{ij})_{(m+1)}+\sum_{l=1}^{m}(e_{ij}^{(n)})_{(l)}$ for all $i<j$, and $(e^{(n)}_{ii})_{(m+1)}-\lambda_{i}$ for all $i$. Then the inclusion into the first $m$ factors $\Un^{\otimes m}\hookrightarrow \Un^{\otimes m+1}$ composed with projection $p_{m}:\Un^{\otimes m+1}\rightarrow \overline{\Un^{\otimes m+1}}$ gives the isomorphism $q_{m}:\Un^{\otimes m}\rightarrow \overline{\Un^{\otimes m+1}}$. Then the operator $\widetilde{\mathfrak{D}}_{n}$ is obtained by applying the homomorphism $q_{m}^{-1}\circ p_{m}: \Un^{\otimes m+1}\rightarrow \Un^{\otimes m}$ to each coefficient of $\overline{\mathfrak{D}}_{n}$.  

    Let $\beta_{i}(u)\in \bigl(\Un^{\otimes m}\bigr)((u^{-1}))$, $i=0\lc n$, be the following coefficients of the differential operator $\widetilde{\mathfrak{D}}_{n}$:
    \[\widetilde{\mathfrak{D}}_{n}=\sum_{i=1}^{n}\beta_{i}(u)(u\partial_{u})^{n-i},\]
    and let $\overline{\beta}_{i}(u)$, $i=0\lc n$ be the images of $\beta_{i}(u)$ in $\bigl(\End S^{(n)}(b)[a]\bigr)((u^{-1}))$. By a lengthy but straightforward computation, using the procedure of obtaining $\widetilde{\mathfrak{D}}_{n}$ described above, we computed $\overline{\beta}_{1}(u)$ and $\overline{\beta}_{2}(u)$ explicitly and checked that for each $i=1\lc m$, we have
    \[\frac{1}{\xi_{i}}\Res_{u=\xi_{i}}\left(\frac{1}{2}\overline{\beta}^{2}_{1}(u)-\overline{\beta}_{2}(u)\right)=\overline{H}^{(n)}_{i}(z,\xi)+\frac{1}{2}b^{2}_{i}.\]

    By the definition of $\mathfrak{D}_{n}^{tG}(b,\xi;a,\lambda^{(n)})$ and the construction of $\overline{\beta}_{i}(u)$, we have $u^{n}\mathfrak{D}_{n}^{tG}(b,\xi;a,\lambda^{(n)})=\sum_{i=1}^{n}\overline{\beta}_{i}(u)(u\partial_{u})^{n-i}$. Comparing the last formula with formula \eqref{Prop 4.2}, we get
    \[\frac{1}{\xi_{i}}\Res_{u=\xi_{i}}\left(\frac{1}{2}\overline{\beta}^{2}_{1}(u)-\overline{\beta}_{2}(u)\right)=\frac{1}{\xi_{i}}\Res_{u=\xi_{i}}\left(\frac{1}{2}C^{2}_{1}(u)-C_{2}(u)\right),\]
    and the theorem is proved.
\end{proof}

Recall the map $\Psi^{gr}_{(\lambda)}$ introduced for each $(\xi,\lambda^{(n)})$ in some open subset $U\in\C^{n+m}$ at the end of the previous subsection. Combining Theorem \ref{Hamiltonians through coefficients} and relation \eqref{deg duality}, we get the following corollary:
\begin{cor}\label{deg duality for Hamiltonians}
    For each $(\xi,\lambda^{(n)})\in U$ and $i=1\lc m$, we have
    \[\widehat{\Psi^{gr}_{(\lambda)}}\bigl(\overline{G}^{(m)}_{i}(z,\xi)\bigr) = \overline{H}^{(n)}_{i}(z,\xi).\]
\end{cor}

Now, comparing the last formula with formula \eqref{ancient duality}, we can prove the following.
\begin{lem}\label{eigenvectors match}
There exists an open subset $\widetilde{U}$ of $U$ such that for every $(\xi,\lambda^{(n)})\in\widetilde{U}$, the following holds:
    \begin{enumerate}
        \item Let $v\in S^{(m)}(a)[b]$ be a common eigenvector of the algebra $\mathcal{B}_{m}^{XXX}(a,\lambda^{(n)};b,\xi)$. Then $\Psi (v) = \varkappa\, \Psi^{gr}_{(\lambda)} (v)$ for some nonzero $\varkappa\in\C$.
        \item Let $w\in S^{(n)}(b)[a]$ be a common eigenvector of the algebra $\mathcal{B}_{n}^{tG}(b,\xi;a,\lambda^{(n)})$. Then $\Psi^{-1} (w) = \widetilde{\varkappa}\, (\Psi^{gr}_{(\lambda)})^{-1} (w)$ for some nonzero $\widetilde{\varkappa}\in\C$.
    \end{enumerate}
\end{lem}
\begin{proof}
    We will prove part (1). Part (2) can be proved similarly. 
    
    Let $v\in S^{(m)}(a)[b]$ be a common eigenvector of the algebra $\mathcal{B}_{m}^{XXX}(a,\lambda^{(n)};b,\xi)$. By Proposition \ref{XXX}, $v$ is a common eigenvector of $gr_{m-i}(D_{ij})$, $i=1\lc m$, $l=1\lc n$. By Theorem \ref{Hamiltonians through coefficients}, this implies that $v$ is a common eigenvector of the dynamical Hamiltonians $\overline{G}^{(m)}_{i}(z,\xi)$, $i=1\lc m$. Denote the corresponding eigenvalues by $g_{i}$, $i=1\lc m$. Then, using Corollary \ref{deg duality for Hamiltonians}, we have
    \[\overline{H}^{(n)}_{i}(z,\xi)\Psi^{gr}_{(\lambda)} (v) = \Psi^{gr}_{(\lambda)} (\overline{G}^{(m)}_{i}(z,\xi)\,v) = g_{i} \Psi^{gr}_{(\lambda)}(v).\]

    Similar computation, but now using formula \eqref{ancient duality} instead of Corollary \ref{deg duality for Hamiltonians}, gives
    \[\overline{H}^{(n)}_{i}(z,\xi)\Psi(v) =  g_{i} \Psi(v).\]
    Therefore, it is left to prove that there exists an open subset $\widetilde{U}$ of $U$ such that for every $(\xi,\lambda^{(n)})\in\widetilde{U}$, common eigenspaces of the trigonometric Gaudin Hamiltonians $\overline{H}^{(n)}_{i}(z,\xi)$, $i=1\lc m$, are one-dimensional. 

    Notice that the images of the elements $(e_{ii}^{(n)})_{(j)}$, $i=1\lc n$, $j=1\lc m$ in $\End (S^{(n)}(b)[a])$ are simultaneously diagonalizable with one-dimensional common eigenspaces and with integer eigenvalues. Thus, if $z_{1}\lc z_{n}$ are linearly independent over $\Z$, the images of the elements $\sum_{i=1}^{n}z_{i}(e_{ii}^{(n)})_{(j)}$, $j=1\lc m$ in $\End (S^{(n)}(b)[a])$ are simultaneously diagonalizable with one-dimensional common eigenspaces. Therefore, the common eigenspaces of the operators $\overline{H}^{(n)}_{i}(z,\xi)$, $i=1\lc m$, are one-dimensional provided that $z_{1}\lc z_{n}$ are sufficiently large positive numbers linearly independent over $\Z$, which implies the existence of $\widetilde{U}$. 
\end{proof}

\begin{prop}\label{Psi match}
    For any $(\xi,\lambda^{(n)})\in\widetilde{U}$ and any $X\in\mathcal{B}_{m}^{XXX}(a,\lambda^{(n)};b,\xi)$, we have $\widehat{\Psi}(X)= \widehat{\Psi^{gr}_{(\lambda)}}(X)$.
\end{prop}
\begin{proof}
    As was mentioned in the proof of Lemma \ref{diagonalization}, for $(\xi,\lambda^{(n)})\in U$, the algebra $\mathcal{B}_{n}^{tG}(b,\xi;a,\lambda^{(n)})$ is diagonalizable.    

    Take $(\xi,\lambda^{(n)})\in\widetilde{U}\subset U$. Let $v_{1}\lc v_{N}$ be a common eigenbasis of $\mathcal{B}_{n}^{tG}(b,\xi;a,\lambda^{(n)})$. Then by Lemma \ref{eigenvectors match}, part (2), for any $i=1\lc N$, there exists nonzero $\widehat{\varkappa}_{i}\in\C$ such that
    \begin{equation}\label{F1}
        \Psi^{-1}(v_{i})=\widetilde{\varkappa}_{i}\,(\Psi^{gr}_{(\lambda)})^{-1}(v_{i}).
    \end{equation}

    Notice that $(\Psi^{gr}_{(\lambda)})^{-1}(v_{i})$, $i=1\lc N$, are common eigenvectors of $\mathcal{B}_{m}^{XXX}(a,\lambda^{(n)};b,\xi)$. Then, by Lemma \ref{eigenvectors match}, part (1), for each $i=1\lc N$, we can find nonzero $\varkappa_{i}\in\C$ such that 
    \begin{equation}\label{F2}
        \Psi((\Psi^{gr}_{(\lambda)})^{-1}(v_{i}))=\varkappa_{i}\,v_{i}.
    \end{equation}
    Combining the last formula with \eqref{F1}, we get $\widetilde{\varkappa}_{i}=\varkappa_{i}^{-1}$.

    Now, consider an arbitrary $X\in\mathcal{B}_{m}^{XXX}(a,\lambda^{(n)};b,\xi)$. In the following computation, we use formula \eqref{F1} on the second step, and we use formula \eqref{F2} and the fact that $X(\Psi^{gr}_{(\lambda)})^{-1}\,v_{i}$ is proportional to $(\Psi^{gr}_{(\lambda)})^{-1}\,v_{i}$ on the third step:
    \begin{equation*}
        \begin{split}
            \widehat{\Psi}(X)v_{i} & = \Psi X\Psi^{-1}\,v_{i} = \varkappa_{i}\Psi X (\Psi^{gr}_{(\lambda)})^{-1}\, v_{i} = \\
            & = \widetilde{\varkappa}_{i}\varkappa_{i}\, \Psi^{gr}_{(\lambda)} X (\Psi^{gr}_{(\lambda)})^{-1}\, v_{i} = \widehat{\Psi^{gr}_{(\lambda)}}(X) v_{i}.
        \end{split}
    \end{equation*}
    Since $v_{1}\lc v_{N}$ is a basis, we deduce $\widehat{\Psi}(X)= \widehat{\Psi^{gr}_{(\lambda)}}(X)$.
\end{proof}

Now we can easily prove Theorem \ref{main3}. Indeed, existence of $\widetilde{A}_{ij}^{(m)}$ follows from Proposition \ref{XXX} and the fact that the difference operator $D_{m}^{XXX}(a,\lambda^{(n)}-\boldsymbol{n};b,\xi)$ can be obtained from $D_{m}^{XXX}(a,\lambda^{(n)};b,\xi)$ using the substitution $u\mapsto u-n$. Existence of $\widetilde{A}_{ji}^{(n)}$ follows from Proposition \ref{tG} and the relation 
\begin{equation}\label{u partial}
u^{j}\partial_{u}^{j}=\prod_{i=1}^{j}(u\partial_{u}-j+1).
\end{equation}

If $(\xi,\lambda^{(n)})\in\widetilde{U}$, by relation \eqref{deg duality} and Proposition \ref{Psi match}, for each $i=1\lc m$, $j=1\lc n$, we have $\widehat{\Psi}(gr_{m-i}\,D_{ij}) = gr_{m-i}\,C_{ij}$. Comparing formulas \eqref{coef A tilde m}, \eqref{coef A tilde n}, \eqref{Prop 4.2}, and \eqref{gr D ij}, and using relation \eqref{u partial} again, we see that $\widetilde{A}^{(m)}_{ij}$ is expressed polynomially through $gr_{m-k}(D_{kl})$ in the same way as $\widetilde{A}^{(n)}_{ji}$ is expressed polynomially through $gr_{m-k}(C_{kl})$. Therefore, for $(\xi,\lambda^{(n)})\in\widetilde{U}$, we have $\widehat{\Psi}(\widetilde{A}^{(m)}_{ij})=\widetilde{A}^{(n)}_{ji}$. Since $\widetilde{A}^{(m)}_{ij}$ and $\widetilde{A}^{(n)}_{ji}$ depend polynomially on $\xi$ and $\lambda^{(n)}$, the last equality holds for any $(\xi,\lambda^{(n)})\in\C^{n+k}$, and therefore, Theorem \ref{main3} is proved.


\bibliographystyle{alpha}
\bibliography{mybibliography1}

\end{document}